\documentclass[12 pt, reqno]{amsart}
\usepackage[utf8]{inputenc}
\usepackage[margin = 1.1 in]{geometry}

\usepackage{amsmath}
\usepackage{amssymb}
\usepackage{amsthm}
\usepackage{mathrsfs}
\usepackage{mathtools}
\usepackage{tikz}
\usepackage{cleveref}
\usepackage{tikz-cd}
\usepackage{mathtools}
\usepackage{verbatim}
\usepackage{appendix}
\usepackage{xcolor}
\usepackage{longtable}
\usepackage{graphicx}

\newcommand{\rood}[1]{\ignorespaces} 

\newtheorem{thm}{Theorem}[section]
\crefname{thm}{Theorem}{Theorems}
\newtheorem{cor}[thm]{Corollary}
\newtheorem{prop}[thm]{Proposition}
\crefname{prop}{Proposition}{Propositions}
\newtheorem{lem}[thm]{Lemma}
\crefname{lem}{Lemma}{Lemmas}
\newtheorem{clm}[thm]{Claim}

\newtheorem{defn}[thm]{Definition}

\crefname{defn}{Definition}{Definitions}
\newtheorem{prin}[thm]{Principle}
\crefname{prin}{Principle}{Principles}
\theoremstyle{definition}
\newtheorem{exmp}[thm]{Example}

\newtheorem{notn}[thm]{Notation}

\newtheorem{rmk}[thm]{Remark}

\newtheorem*{ack*}{Acknowledgements}

\newcommand{\on}{\operatorname}
\newcommand{\mb}{\mathbb}

\newcommand{\mc}{\mathcal}
\newcommand{\ms}{\mathscr}

\newcommand{\wt}{\widetilde}
\renewcommand{\P}{\mb{P}}
\newcommand{\Orb}{\mathsf{Orb}}
\renewcommand{\k}{\mathbb{C}}
\newcommand{\p}{\mathsf{p}}

\title{Equivariant Degenerations of Plane Curve Orbits}
\author{Mitchell Lee, Anand Patel, Dennis Tseng}
\date{\today}

\begin{document}

\maketitle
\begin{abstract}
  
  In a series of papers, Aluffi and Faber computed the degree of the
  $GL_3$-orbit closure of an arbitrary plane curve. We attempt to
  generalize this to the equivariant setting by studying how orbits
  degenerate under some natural specializations, yielding a fairly
  complete picture in the case of plane
  quartics. 

\end{abstract}

\section{Introduction}

Let $V$ be an $(r+1)$-dimensional vector space, and let
$F \in \on{Sym}^{d}V^\vee$ be a non-zero degree $d$ homogeneous form
on $V$. The form $F$ naturally produces two varieties, first the
$GL(V)$-orbit closure

\[\Orb(F) \subset \on{Sym}^{d}V^\vee,\] and secondly its
projectivization
\[\P \Orb(F) \subset \P\on{Sym}^{d}V^\vee.\]  The relationship between the
geometry of $\P \Orb(F)$ and the geometry of the hypersurface $X$
defined by $\{F=0\}$ remains mostly mysterious.

Consider, for example, the enumerative problem of computing the degree
of $\P \Orb(F)$. The analysis of the degree of $\P \Orb(F)$ was
carried out for the first two cases $r=1,2$ in a series of remarkable
papers of Aluffi and Faber
\cite{AFpoint,AF93,AFsmall1,AFsmall2,AF00,AFpgl1,AFpgl2}.  Aluffi and
Faber's computation in the special case $r=2, d=4$ of quartic plane
curves yields the highly non-trivial enumerative consequence: {\sl In
  a general $6$-dimensional linear system of quartic curves, a general
  genus $3$ curve arises $14280$ times.}

Our starting point is to interpret the calculation of the degree of
$\P \Orb(F)$ as equivalent to the computation of the fundamental class
$[\P \Orb(F)] \in A^{\bullet}(\P \on{Sym}^{d}V^{\vee})$. Our extra
contributions stem from one very simple observation. Since
$\P \Orb(F)$ is evidently preserved by the action of $GL(V)$, there is
a natural {\sl equivariant} extension of the problem: Compute the {\sl
  equivariant} fundamental class
\[[\P \Orb(F)]_{GL(V)} \in A^{\bullet}_{GL(V)}(\P
  \on{Sym}^{d}V^{\vee}).\]

In elementary terms, beginning with a rank $r+1$ vector bundle
$\mc{V}$ over a base $B$, the class $[\P \Orb(F)]$ encodes the
universal expressions in the chern classes
$c_1\mathcal{V}, ..., c_{r+1}\mathcal{V}$ appearing in the fundamental
class of the relative orbit closure cycle
$(\P \Orb(F))_{\mathcal{V}} \subset \P
\on{Sym}^{d}\mathcal{V}^{\vee}$. This larger equivariant setting
encodes the solution to many more challenging enumerative
problems. For instance, our analysis of the case of quartic plane
curves yields the enumerative consequence: {\sl A general genus $3$
  curve appears $510720$ times as a $2$-plane slice of a fixed general
  quartic threefold.}

To date, very few equivariant classes $[\P \Orb(F)]_{GL(V)}$ have been
determined. For quadric hypersurfaces ($d=2$), the class
$[\P \Orb(F)]_{GL(V)}$ is determined by the rank of the quadric
hypersurface $F=0$ and its computation amounts to the Porteous formula
for symmetric maps \cite{HT84}. For another class of examples, the
authors' work with H. Spink in \cite{FML} provides access to
$[\P \Orb(F)]_{GL(V)}$ when $F=0$ defines a hyperplane arrangement. In
this paper, we study the frontier case $r=2$ of plane curves. As in
\cite{FML}, our strategy will be to {\sl specialize} the orbit
$\P \Orb(F)$ until it breaks into a union of other orbits
$\P \Orb(F_i)$ whose classes are more directly computable. To execute
this strategy, we initiate a detailed study of how orbits of plane
curves behave under some common degenerations.  This is our main
purpose in the paper.
 
We will summarize the results of this study in the next subsection,
but we emphasize here that the overall picture in the case of quartic
plane curves is quite neat and beautiful -- we deduce clean
relationships among different classes $[\P \Orb(F)]_{GL(V)}$ for $F$
ranging over several types of quartic plane curves possessing special
geometric properties.

Finally, because the computation of equivariant orbit classes does not
have a strong presence in the literature, we have written an appendix
containing the $r=1$ case (points on a line) and the $r=2, d=3$ case
of cubic plane curves. As a refreshing demonstration of alternate
techniques, the case of points on a line is done in two independent
ways: one by specializing the results in \cite{FML} and the other by
applying the Atiyah-Bott formula to the resolution of the orbit given
in \cite{AFpoint}.

\subsection{Summary of Degenerations}
\label{sec:degeneration}
When studying degenerations, it is more convenient to weight orbits
$\P \Orb(F)$ by the number of linear automorphisms of the hypersurface
$X$ given by $F=0$.  We will write $\Orb(F)$ and $\Orb(X)$
interchangeably (and similarly for their projectivized versions).  In
this section, we will exclusively be concerned with the case of plane
curves, $r=2$.

If $C_t$ is a family of smooth plane curves specializing at $t=0$ to a
curve $C_0$ possessing nodes and cusps, we get a corresponding
specialization of the (weighted) orbit closures $\P \Orb(C_{t})$ to
some union of $GL(V)$-invariant varieties. \Cref{nodecusp} gives a
complete description of the new orbits appearing in the flat limit as
$t \to 0$. To illustrate this theorem, we will describe what happens
in the special case where the curve acquires a single node or a single
cusp. The general case is essentially a sum of contributions
corresponding to each node or cusp.
\subsubsection{Acquiring a node}
\label{subsec:node}
If $C_t$ acquires a single node in the limit $C_0$, then as a limit of
weighted orbits, one obtains the obvious weighted orbit
$\P \Orb(C_{0})$ along with one other weighted orbit,
$\P \Orb(C_{BN})$, which occurs with multiplicity 2.
\label{sec:summarydeg}
\begin{center}
  \includegraphics[scale=.2]{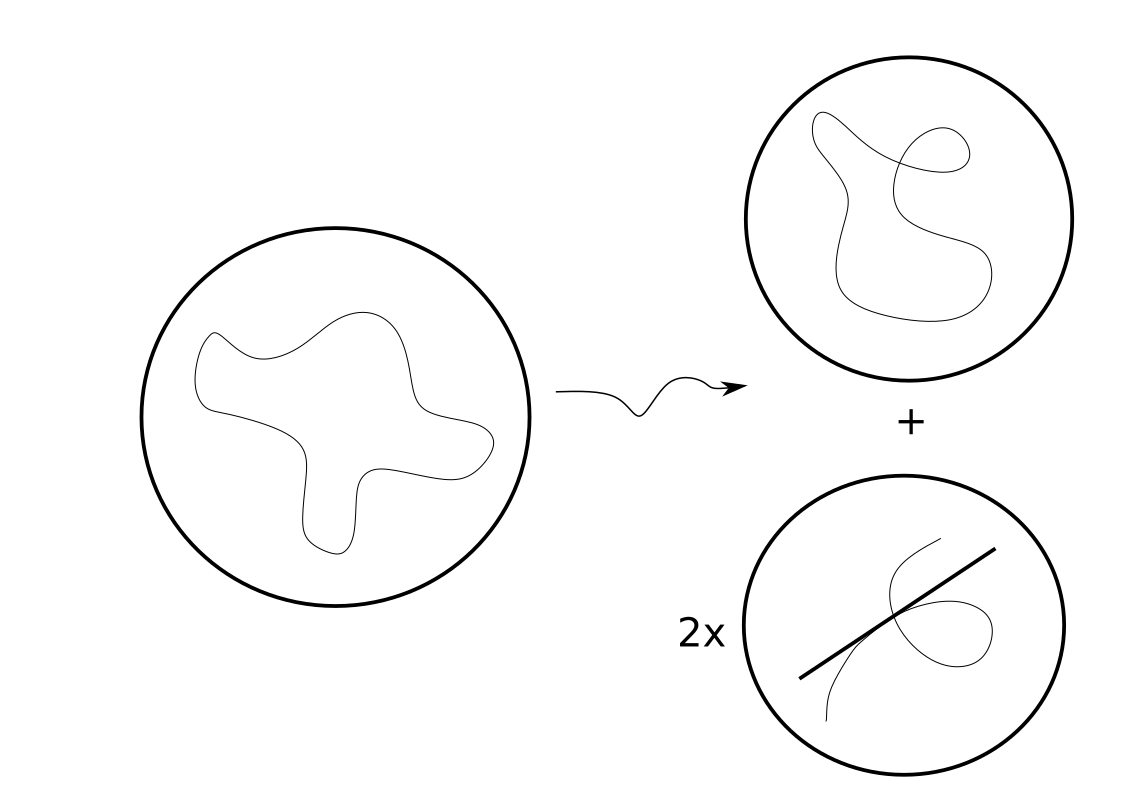}
\end{center}
The curve $C_{BN}$ is the union of a nodal cubic with one of its two
tangent lines (branches) at the node, the line taken with multiplicity
$(d-3)$.
\subsubsection{Acquiring a cusp}
If $C_t$ acquires a single cusp in the limit $C_0$, then as a limit of
weighted orbits, one obtains the weighted orbit of the cuspidal curve
$C_{0}$ along with another weighted orbit, $\P \Orb(C_{\on{flex}})$.
\begin{center}
  \includegraphics[scale=.2,trim={0 0 11cm
    0},clip]{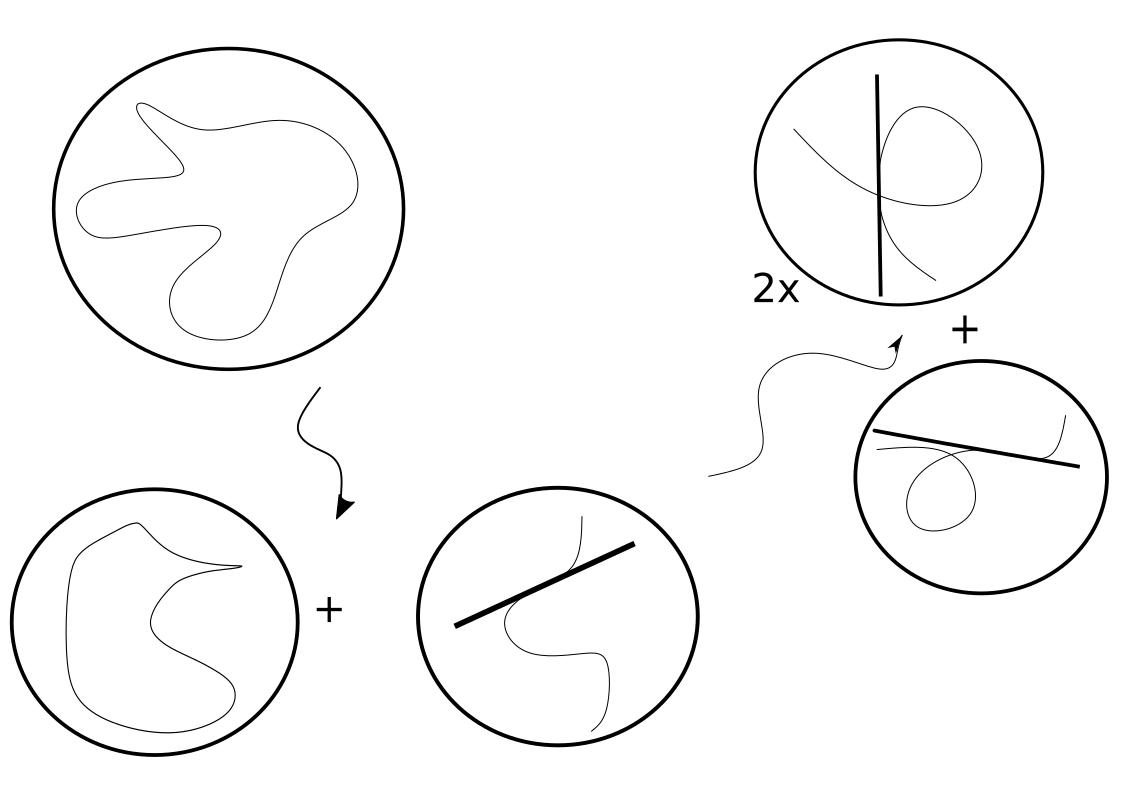}\hspace{1in}
  \includegraphics[scale=.2,trim={11cm 0 0
    0},clip]{ImageFiles/cuspsiblingarXiv.png}
\end{center}
The curve $C_{\on{flex}}$ is a smooth cubic (with general
$j$-invariant) union a $(d-3)$-fold flex line. We can degenerate the
weighted orbit $\P \Orb(C_{\on{flex}})$ even further to get the
weighted orbit of $C_{BN}$ from before (with multiplicity 2) together
with the weighted orbit of a new curve $C_{AN}$. $C_{AN}$ is the union
of a nodal cubic with one of its flex lines (at a smooth point), with
the line taken $(d-3)$ times.

In either case, we deduce that the equivariant class
$[ \P \Orb({C_{t}})]_{GL(V)}$, for $t$ general, is a particular
combination of $[ \P \Orb({C_{0}})]_{GL(V)}$ and two specific classes
$[ \P \Orb({C_{BN}})]_{GL(V)}$, and $[ \P \Orb({C_{AN}})]_{GL(V)}$.

\subsubsection{Splitting off a line}
The equivariant class of the orbit closure of a union of lines can be
deduced using the results of \cite{FML}.  In light of this, it is
natural to want to study the degeneration where a degree $d$ plane
curve specializes to a union of $d$ general lines. We will show that
if $C_t$ is a general family of curves such that $C_0$ is a general
union of lines, then in the flat limit of orbit closures, beyond the
obvious orbit closure $\P \Orb({C_{0}})$ we find $d$ other orbits,
each being the weighted orbit of a general irreducible plane curve
possessing a multiplicity $d-1$ point.
\begin{center}
\includegraphics[scale=.2]{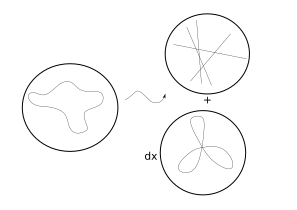}
\end{center}
More generally, we also study specializations where $C_0$ is a union
of a general degree $e$ curve along with $d-e$ general lines (see
\Cref{linespecialization}).

\begin{rmk}
  For any reader experienced in the art of degeneration, the fact that
  the curves with $(d-1)$-fold points arise in the limit is not a
  surprise.  The difficulty lies in showing that all new orbits are
  accounted for.
\end{rmk}
\subsubsection{Degeneration to the Double Conic}
\label{sec:doubleconic}
Next suppose $C_t$ is a family of general plane quartics specializing
to a double conic -- a classic example in the study of moduli of
curves.  Since a double conic has a large stabilizer group under the
action of $PGL(V)$, its orbit closure will not appear as an
irreducible component of the $t \to 0$ flat limit of orbit
closures. We will show (in \Cref{smooth}) that in this situation, the
weighted orbit of $C_t$ specializes to $8$ times the weighted orbit of
a particular rational quartic curve possessing an $A_6$ singularity.

\begin{center}
    \includegraphics[scale=.3]{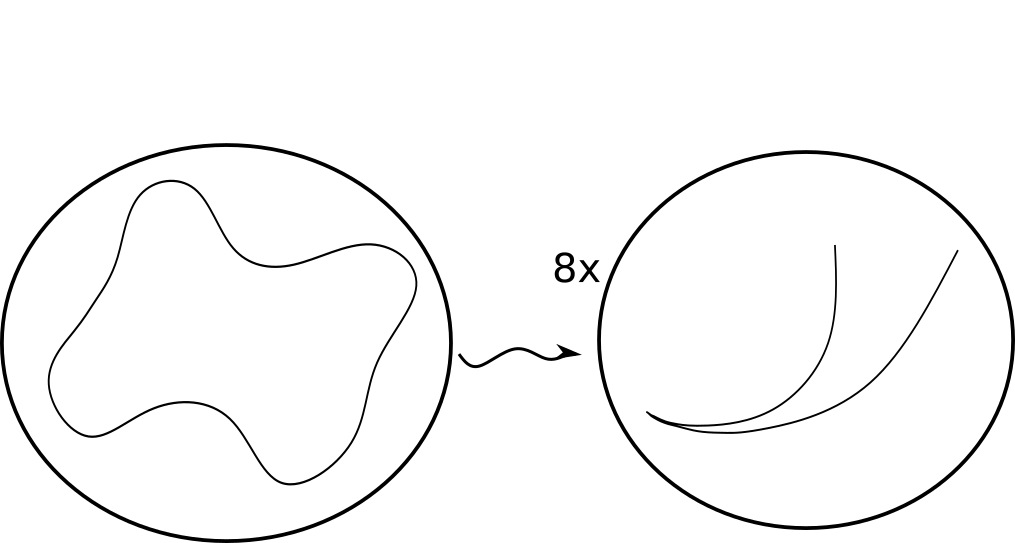}
\end{center}

It is striking that the limit has such a clean answer, consisting
set-theoretically of the orbit closure of {\sl the} irreducible
quartic plane curve with an $A_6$ singularity. This fact is related to
the question of which planar quartics account for a {\sl general}
hyperelliptic curve after applying semistable reduction. The
multiplicity 8 we obtain corresponds to the 8 Weierstrass points on a
genus 3 hyperelliptic curve. For literature relevant to this circle of
ideas, see \cite{P74,H00,CL13,F14}.

\subsection{Orbit classes of quartic curves}
\label{sec:quartic}

In the specific setting of quartic curves, we find that the orbit
class of an {\sl arbitrary} smooth quartic can be deduced in a direct
way from the orbit classes of special quartics with $A_6$ and $E_6$
singularities.  We have already explained the relation with curves
having an $A_6$ singularity above.  By adapting an idea of Aluffi and
Faber \cite[Theorem IV(2)]{AF93P}, in \Cref{sec:hyperflex} we
specialize the orbit closure of a general quartic plane curve to the
orbit of a smooth quartic plane curve possessing a hyperflex. In the
limit, the orbit closure of a particular rational quartic (denoted
$C_{E_{6}}$ in the text) possessing an $E_6$ singularity appears (with
multiplicity twice the number of hyperflexes of the limiting smooth
quartic).

Ultimately, we deduce that the orbit class $[\P \Orb({F})]_{GL(V)}$ of
an {\sl arbitrary} smooth quartic $F=0$ is expressible solely in terms
of orbit classes of two specific rational curves having an $A_6$ or
$E_6$ singularity.  These are the curves $C_{A_{6}}$ and $C_{E_{6}}$
in the text. From here, we arrive at explicit formulas for orbit
classes by invoking Kazarian's work \cite{Kazarian} on counting $A_6$
and $E_6$ singularities in families of curves.

We also compute the equivariant classes of orbit closure for many
singular quartics using the degenerations in
\Cref{sec:degeneration}. (For those readers curious about the lack of
the presence of $D_6$ singularities, the curves with a $D_6$
singularity arise when specializing to a node as in
\Cref{subsec:node}.)

We summarize all computations in the table in \Cref{quartictable}.
Let us explain how to read the table.
For a plane curve $C\subset\mb{P} V$ with an $8$-dimensional $PGL(V)$
orbit, the expressions $\p_{C}$ are defined to be the
$GL(V)$-equivariant classes $[\Orb(C)]_{GL(V)}$ multiplied by the
number of $PGL(V)$-automorphisms of $C$. We note here that the classes
$[\mb{P}\Orb(C)]_{GL(V)}$ are related to $[\Orb(C)]_{GL(V)}$ by a
simple substitution --see \Cref{projectivethom}.

\begin{figure}[p]
\begin{longtable}{p{6.5cm} | p{9cm}}
  Quartic Plane Curve $C$ & $\p_C(c_1,c_2,c_3)$\\\hline
  $C_{A_6}: (x^2 + yz)^2 + 2yz^3 = 0$ & 
                                        $ 3\cdot112 (9 c_1^3 + 12 c_1 c_2 - 11 c_3) (2 c_1^3 + c_1 c_2 + c_3)$\\
  $C_{D_6}: z(xyz + x^{3} + z^{3})$ &$3\cdot 64 (18 c_1^6 + 33 c_1^4 c_2 + 12 c_1^2 c_2^2 - 85 c_1^3 c_3 - 11 c_1 c_2 c_3 - 7 c_3^2)$\\
  $C_{E_6}: y^3z + x^4 + x^2y^2 = 0$ &
                                       $2\cdot 48 (2 c_1^3 + c_1 c_2 + c_3) (9 c_1^3 - 6 c_1 c_2 + 7 c_3)$\\
  $C_{AN}$: Nodal cubic union flex line & $2\cdot 192 (18 c_1^6 + 33 c_1^4 c_2 + 12 c_1^2 c_2^2 + 19 c_1^3 c_3 - 7 c_1 c_2 c_3 - 
                                          35 c_3^2)$\\
  $C_{\on{flex}}$: smooth cubic union flex line & $\p_{C_{AN}}+2\p_{D_6}$\\
  $Q$: Quadrilateral &  
                       $24\cdot 16 (18 c_1^6 + 33 c_1^4 c_2 + 12 c_1^2 c_2^2 + 131 c_1^3 c_3 + 
                       153 c_1 c_2 c_3 - 147 c_3^2)$\\
  $C_{D_{4}}$: a general curve with $D_4$ singularity & 
                                                        $\frac{1}{4}(8\p_{C_{A_{6}}} - \p_{Q})$
  \\
  Two lines plus conic & $\p_{Q} + 2\p_{C_{D_{4}}}$\\
  A line plus a general cubic & 
                                $\p_{Q} + 3\p_{C_{D_{4}}}$
  \\
  Quartic with $\delta$ nodes and $\kappa$ cusps and no hyperflexes& $8\p_{C_{A_{6}}}-2\delta \p_{C_{D_6}}-\kappa \p_{C_{\on{flex}}}$\\
  A smooth quartic with $n$ hyperflexes &  $8\p_{C_{A_{6}}}-n\p_{C_{E_6}}$\\
  General smooth quartic & $8\p_{C_{A_{6}}}$
\end{longtable}
\caption{Equivariant classes of orbits of quartic plane curves}
\label{quartictable}
\end{figure}

\subsubsection{Application: plane sections of a quartic threefold}
To see how to use the expressions $\p_C$, we provide the calculation
of $510720$ mentioned earlier in the introduction.

Starting with a general smooth quartic threefold
\[X\subset \mathbb{P}^4,\]
one obtains a rational map
\begin{align*}
   \Phi:  \mb{G}(2,4)\dashrightarrow \overline{M_3}
\end{align*}
sending a general $2$-plane $\Lambda\subset\mathbb{P}^4$ to the moduli
point of the plane quartic $C := X \cap \Lambda$. Our calculation of
the equivariant class $\p_{C}$ for $C$ a general quartic implies:
\begin{cor}
\label{numberplanesection}
If $X$ is general, the map $\Phi$ has degree 510720.
\end{cor}

\begin{figure}[p]
\begin{longtable}{p{6.5cm} | p{9cm}}
  Quartic Plane Curve $C$ & ($\#\on{Aut}(C)\cdot \#$ planar sections of general quartic threefold)\\\hline
  $C_{A_6}: (x^2 + yz)^2 + 2yz^3 = 0$ & 
                                        $ 3\cdot21280$\\
  $C_{D_6}: z(xyz + x^{3} + z^{3})$ &$3\cdot 7040$\\
  $C_{E_6}: y^3z + x^4 + x^2y^2 = 0$ &
                                       $2\cdot 4800$\\
  $C_{AN}$: Nodal cubic union flex line & $2\cdot 36480$\\
  $C_{\on{flex}}$: smooth cubic union flex line & $2\cdot 57600$\\
  $Q$: Quadrilateral &  
                       $24\cdot 5600$\\
  $C_{D_{4}}$: a general curve with $D_4$ singularity & 
                                                        $94080$
  \\
  Two lines plus conic & $322560$\\
  A line plus a general cubic & 
                                $416640$
  \\
  Quartic with $\delta$ nodes and $\kappa$ cusps and no hyperflexes& $510720-2\delta (3\cdot 7040)-\kappa (2\cdot 57600)$\\
  A smooth quartic with $n$ hyperflexes &  $510720-n(2\cdot 4800)$\\
  General smooth quartic & $510720$
\end{longtable}
\caption{Number of times we see a particular curve as a planar section
  of a quartic threefold with \emph{specified moduli}.}
\label{sectiontable}
\end{figure}

\begin{proof}[Proof of \Cref{numberplanesection}]
  We find the degree of $\Phi$ directly by choosing a general quartic
  plane curve $C\subset \mb{P}^2$ and counting the number of
  $2$-planes $\Lambda$ such that $X\cap \Lambda$ is isomorphic to $C$.

  Let $G$ be a quartic homogenous form cutting out a general quartic
  threefold $X\subset \mb{P}^4$, and let $\pi: \mc{S}\to \mb{G}(2,4)$
  denote the rank $3$ tautological subbundle over the
  Grassmannian. The form $G$ defines a section of
  $\ms{O}_{\mb{P}(\mc{S})}(4)$ on $\mb{P}(\mc{S})$, which in turn
  induces a section $s: \mb{G}(2,4)\to \on{Sym}^4\mc{S}^{\vee}$. Let
  $(\Orb(C))_{\mc{S}} \subset \on{Sym}^4\mc{S}^{\vee}$ be the
  relative orbit.  Since $G$ is general, the section $s$ will
  intersect $(\Orb(C))_{\mc{S}}$ only in the interior of the relative
  orbit. Since $C$ is general the intersection will consists of
  reduced points by generic reducedness in characteristic zero. In
  this paper, we will assume the characteristic is at least 7 (see
  \Cref{basefield}), and a standard transversality argument which we
  omit shows that the intersection is also still reduced in this
  case. Therefore, $\on{Image}(s)$ and $(\Orb(C))_{\mc{S}} $ are
  smooth at the reduced scheme
  $\on{Image}(s)\cap (\Orb(C))_{\mc{S}} $, and
  $\deg(\Phi)=\int_{\mb{G}(2,4)}s^{*}[(\Orb(C))_{\mc{S}} ]$.
  Expanding the formula for $\p_C$, we get
\begin{align*}
  &48384 c_{1}(S)^6 + 88704 c_{1}(\mc{S})^4 c_{2}(\mc{S}) + 32256 c_{1}(\mc{S})^2 c_{2}(\mc{S})^2 - 34944 c_{1}(\mc{S})^3 c_{3}(\mc{S})\\ + 
  &2688 c_{1}(\mc{S}) c_{2}(\mc{S}) c_{3}(\mc{S}) - 29568 c_{3}(\mc{S})^2.
\end{align*}

By evaluating on the Grassmannian, we conclude:
\begin{align*}
  \deg \Phi =  48384 \cdot 5 + 88704 \cdot 3 + 32256 \cdot 2 - 34944 + 2688 - 29568 = 510720, 
\end{align*}
as claimed.
\end{proof}

Beyond computing the number of times a {\sl general} planar quartic
appears on $X$, our methods also apply to various singular quartics as
well. The results are summarized in \Cref{sectiontable}.

\begin{exmp}
  The number of tricuspidal curves arising as a section of a quartic
  threefold is 27520 by applying Kazarian's theory of
  multisingularities. More precisely, the number can in principle be
  deduced from \cite[Section 8]{K03}, but the formula for $2$-planes in
  $\mb{P}^4$ meeting a degree $d$ hypersurface in a curve with three
  cusps can be found on Kazarian's website. From \Cref{sectiontable},
  we get $510720 - 3\cdot 2\cdot57600=6\cdot 27520$, accounting for
  the 6 automorphisms of the tricuspidal quartic.  This agrees with
  Kazarian's formula. However, for example, our \Cref{sectiontable}
  also computes the number of 1-cuspidal and 2-cuspidal curves having
  prescribed moduli, which is not covered by the theory of
  multisingularities.
\end{exmp}

\begin{rmk}
  The class $\p_{C}$ for $C$ a general quartic curve, as well as the
  number $510720$ has also been verified independently by the authors
  using the SAGE Chow ring package \cite{ChowRing} to implement the
  resolution used by Aluffi and Faber \cite{AF93} for smooth plane
  curves, though now in the relative setting. However, the
  computations were too cumbersome to verify by hand.
\end{rmk}

\subsection{Related Work}

This paper was heavily influenced and inspired by Aluffi and Faber's
computation of degrees of orbit closures of plane curves of arbitrary
degree.  Zinger also computed the degree of the orbit closure of a
general quartic as a special case of interpolating genus 3 plane
curves with a fixed complex structure \cite{Z05}.

\subsubsection{Planar sections of a hypersurface which have fixed
  moduli}
Counting linear sections of a hypersurface with fixed moduli has been
considered in the case of line sections of a quintic curve \cite{CL08}
and generalized to line sections of hypersurfaces of degree $2r+1$
hypersurfaces in $\mb{P}^r$ \cite{FML} by extending the computations
in the $r=1$ case (see \cite{AFpoint}) to the equivariant setting.

\subsubsection{Counting curves with prescribed singularities}
In addition to Kazarian's work \cite{Kazarian}, there have been
independent efforts to enumerate curves with various types of punctual
singularities, including \cite{BM16,K06,H03}.  

\subsection{Assumptions on the characteristic of the base field}
\label{basefield}

We will assume that our ground field is the complex numbers $\k$ has
characteristic zero.  However, all our techniques should work for
perfect fields of finite characteristic $\geq 7$, though we have not
paid special attention to the matter.

\subsection{Acknowledgements}
We would like to thank Paolo Aluffi and Joe Harris for helpful
conversations. We would also like to thank Carl Lian for a careful
reading and helpful comments.  Finally, we thank an anonymous referee
for helpful suggestions.

\section{Preliminary Definitions, Conventions, and Notation}
\label{sec:definitions}
In this section, we define equivariant generalizations of predegrees
of orbits of hypersurfaces as defined and studied by Aluffi, Faber,
and Tzigantchev \cite{T08,AFpoint,AF93,AF00}. Although we will
ultimately only deal with the case of points on a line, and plane
cubics and quartics, we provide general definitions for clarity and to
set the stage for future work.

\subsection{Conventions}
\label{sec:conventions}

We will work over an arbitrary algebraically closed base field $\k$.
By a {\sl scheme} we will mean a scheme of finite type over $\k$. By a
{\sl variety} we will mean a reduced, though not necessarily
irreducible, scheme. All varieties in the paper will be
quasi-projective.  If $X$ is a variety, then a {\sl subvariety} will
be a closed subscheme of $X$ which is also a variety.  As a rule, the
projectivization of a vector bundle parametrizes $1$-dimensional {\sl
  subspaces}, not quotients.  If $\mathcal{V} \to B$ is a vector
bundle, and if $H$ is the hyperplane class associated to the natural
$\mathcal{O}(1)$ on $\mathcal{V}$, then $H$ satisfies the {\sl Leray}
relation
\[H^{r+1} + c_{1}(\mathcal{V}) \cdot H^{r} + \dots +
  c_{r+1}(\mathcal{V}),\] where $c_{i}$ are the chern classes of
$\mathcal{V}$.

Until we explicitly specialize to the case $r=2$, $V$ will be a fixed
$(r+1)$-dimensional $\k$-vector space.  The action of $GL(V)$ on the
space $\on{Sym}^{d}V^{\vee}$ is as follows: If $F: V \to \k$ is a
degree $d$ form (i.e. an element of $\on{Sym}^{d}V^{\vee}$), and if
$g\in GL(V)$, then $g \cdot F$ is defined to be the form
$F \circ g^{-1}: V \to \k$.  This action descends to an action of
$PGL(V)$ on the projectivization $\P \on{Sym}^{d}V^{\vee}$, which we
will denote by $g \cdot [F]$.  If $X \subset \P V$ is a degree $d$
hypersurface, we will often write $X \in \P \on{Sym}^{d}V^{\vee}$ to
denote $[F]$ where $F$ is any defining equation for $X$.

The action of $GL(V)$ on $\on{End}(V)$ by {\sl post-composition} is as
follows: $g \in GL(V)$ acts on $\gamma \in \on{End}(V)$ via
$ g \circ \gamma$.  (This is to distinguish the action by {\sl
  pre-composition} where $g$ acts by sending $\gamma$ to
$\gamma \circ g^{-1}$.)  Either action descends to an action of
$PGL(V)$ on $\P \on{End}(V)$ in the obvious way.

\subsection{$GL(V)$-equivariant Chow classes}
Suppose $Z \subset \on{Sym}^dV^{\vee}$ is a subvariety preserved by
the action of $GL(V)$, and let
$\P Z \subset \mb{P}(\on{Sym}^dV^{\vee})$ denote its projectivization.

In the presence of a rank $r+1$ vector bundle $\mathcal{V} \to B$ over
a base variety $B$, $Z$ determines a corresponding {\sl relative
  cycle} $Z_{\mathcal{V}} \subset \on{Sym}^d\mathcal{V}^{\vee}$ as
follows: Let $U \subset B$ denote an open set over which $\mathcal{V}$
can be trivialized.  After choosing any particular isomorphism of
$\mathcal{V}|_{U}$ with the product vector bundle $U \times V$ we
obtain a corresponding identification of
$\on{Sym}^d\mathcal{V}^{\vee}|_{U}$ with
$U \times \on{Sym}^dV^{\vee}$. In the latter space, we can take
$U \times Z$, and consider it, via the chosen isomorphism, as a
locally closed subset of $\on{Sym}^d\mathcal{V}^{\vee}$.  Its closure
is denoted $$Z_{\mathcal{V}} \subset \on{Sym}^d\mathcal{V}^{\vee},$$
and we call it the {\sl relative cycle} corresponding to $Z$.
$Z_{\mathcal{V}}$ is well-defined precisely because $Z$ was preserved
by the action of $GL(V)$. In a similar fashion, we define the
projectivized relative cycle
$\P Z_{\mathcal{V}} \subset \P \on{Sym}^d\mathcal{V}^{\vee}$.


Although $Z_{\mc{V}}$ obviously depends on $B$ and $\mc{V}$, its class
in $A^{\bullet}(B)\cong A^{\bullet}(\on{Sym}^d\mc{V}^{\vee})$ is a
universal expression in chern classes of $\mc{V}$ -- this is the
fundamental input we need from the equivariant intersection theory of
Edidin and Graham developed in \cite{EG98}.  In the language of
equivariant intersection theory, this universal class is the
equivariant fundamental class $[Z]_{GL(V)}$ in the equivariant Chow
ring \[A_{GL(V)}^{\bullet}(\on{Sym}^d V^{\vee}).\] The latter ring is
the free polynomial ring $\mathbb{Z}[c_1, \dots, c_{r+1}]$, where
$c_{i}$ are interpreted as the chern classes of the universal vector
bundle $V$ over the classifying stack $\mathbb{B}GL(V)$. To summarize:

\begin{defn}
\label{affinecohomology}
Given $Z$ as above, define $[Z]_{GL(V)}$ to be the polynomial in
$c_1,\ldots,c_{r+1}$ such that the class of $Z_{\mc{V}}$ is
$[Z]_{GL(V)}$ with the chern classes of $\mc{V}$ substituted for
$c_1,\ldots,c_{r+1}$. Equivalently, $[Z]_{GL(V)}$ is the
$GL(V)$-equivariant class of $Z$ in
$A_{GL(V)}^{\bullet}(\on{Sym}^dV^{\vee})\cong
\mb{Z}[c_1,\ldots,c_{r+1}]$.
\end{defn}

In a similar fashion to \Cref{affinecohomology}, we can define an
equivariant Chow class of $\P Z$. As before, there is a single formula
in the chern classes $c_{i}$ of $V$ and the hyperplane class $H$ on
the universal projective bundle $\P \on{Sym}^dV^{\vee}$ ($H$
corresponds to the line bundle
$\mathscr{O}_{\mb{P}(\on{Sym}^d V^\vee)}(1)$) which gives the class of
$[\mb{P}Z_{\mc{V}}]\in A^{\bullet}(\mb{P}(\on{Sym}^d\mc{V}^\vee))$.

\begin{defn}
\label{projectivecohomology}
Given $Z$ as above, define $[\mb{P}Z]_{GL(V)}$ to be the polynomial in
$c_1,\ldots,c_{r+1}$ and $H$ (of degree $\leq {d+r \choose r}-1$ in
$H$) such that the class of $\mb{P}Z_{\mc{V}}$ is $[\mb{P}Z]_{GL(V)}$
with the chern classes of $\mc{V}$ substituted for
$c_1,\ldots,c_{r+1}$ and
$c_{1}\mathscr{O}_{\mb{P}(\on{Sym}^d\mc{V}^\vee)}(1)$ substituted for
$H$. In equivariant language, $[\mb{P}Z]_{GL(V)}$ is the
$GL(V)$-equivariant class of $\P Z$ in
$$A_{GL(V)}^{\bullet}(\mb{P}(\on{Sym}^dV^{\vee}))\cong
\mb{Z}[c_1,\ldots,c_{r+1}][H]/(H^{d+r \choose r}+s_1H^{{d+r \choose
    r}-1}\cdots+s_{{d+r \choose r}}).$$ Here $s_{i}$ is the $i$-th
chern class of $\on{Sym}^d\mc{V}^\vee$, expressed in terms of the
chern classes $c_{i}$.
\end{defn}

It may seem like $[\mb{P}Z]_{GL(V)}$ contains more information than
$[Z]_{GL(V)}$, but the two are related by a simple algebraic
manipulation.  Let $u_1, \dots, u_{r+1}$ denote the formal chern roots
of the universal vector bundle $V$ over $\mathbb{B}GL(V)$ -- in other
words, $c_{i}$ is the $i$-th elementary symmetric expression in
$u_{1}, u_{2}, \dots$. Using the inclusion
$\mb{Z}[c_1,\ldots,c_{r+1}]\hookrightarrow \mb{Z}[u_1,\ldots,u_{r+1}]$
where $c_i$ maps to the $i$-th elementary symmetric function, we can
express $[Z]_{GL(V)}$ as a symmetric polynomial in
$u_1,\ldots,u_{r+1}$ and similarly $[\mb{P}Z]_{GL(V)}$ as a polynomial
in $u_1,\ldots,u_{r+1}$ and $H$ symmetric in the $u_i$'s.  With this
understood, we have:
\begin{prop}[{\cite[Theorem 6.1]{FNR05}}]
  \label{projectivethom}
  We have:
  \begin{align*}
    [Z]_{GL(V)}(u_1,\dots,u_{r+1})&=[\mb{P}Z]_{GL(V)}(u_1,\dots, u_{r+1},0)\\
    [\mb{P}Z]_{GL(V)}(u_1,\dots,u_{r+1},H)&= [Z]_{GL(V)}(u_{1}-\frac{H}{d},\dots,u_{r+1}-\frac{H}{d}).
  \end{align*}
\end{prop}

When $r=1$ and $2$, we will use letters $(u,v)$ for $(u_1,u_2)$ and
$(u,v,w)$ for $(u_1,u_2,u_3)$, respectively.

\subsection{Weighting by automorphism groups}
\label{sec:predegree}

\begin{defn}
  Given $X\subset \mathbb{P}(V)$ a degree $d$ hypersurface, we define
  $$\Orb (X)\subset \on{Sym}^dV^{\vee}$$ to be the $GL(V)$-orbit closure of any
  defining equation $F$ of $X$. Furthermore, we define
  $\on{Aut}(X) \subset PGL(V)$ to be the subgroup consisting of those
  projective automorphisms preserving $X$. Finally, we say that $X$
  has {\sl full dimensional orbit} if the group $\on{Aut}(X)$ is
  finite.
\end{defn}

In this paper, we will exclusively be concerned with hypersurfaces
which have full dimensional orbit.

\begin{defn}
\label{orbitthomdef}
Let $F \in \on{Sym}^dV^{\vee}$ be a degree $d$ homogeneous form
cutting out $X \subset \mathbb{P}(V)$. Then, define
\begin{align*}
  \p_X&:= \begin{cases}
    \#\on{Aut}(X)[\Orb(X)]_{GL(V)} &\qquad\text{if $\#\on{Aut}(X)<\infty$}\\
    0 &\qquad\text{if $\#\on{Aut}(X)=\infty$}.
  \end{cases}\\
  \mathsf{P}_X&:= \begin{cases}
    \#\on{Aut}(X)[\mb{P}\Orb(X)]_{GL(V)} &\qquad\text{if $\#\on{Aut}(X)<\infty$}\\
    0 &\qquad\text{if $\#\on{Aut}(X)=\infty$}.
    \end{cases}   
\end{align*}
\end{defn}

The polynomials $\mathsf{P}_X$ are equivariant generalizations of the
notion of predegree, as defined by Aluffi and Faber
\cite[Definition]{AF93}:
\begin{defn}
\label{predegree}
The {\sl predegree} of a hypersurface $X \subset \P V$ having full
dimensional orbit is $\#\on{Aut}(X)$ times the degree of the orbit
closure $\P \Orb(X)$ in the projective space
$\mb{P} \on{Sym}^{d}V^{\vee}$.  If the orbit of $X$ is not full
dimensional then we define its predegree to be zero.
\end{defn}
\begin{rmk}
  The predegree of a hypersurface $X$ is the coefficient of
  $H^{{d+r \choose r}-(r+1)^2}$ in $\mathsf{P}_{X}$.  Thus, the
  equivariant classes contain much more enumerative data than the
  predegree, namely all the other coefficients.  We will critically
  use the knowledge of the predegree in equivariant arguments.
\end{rmk}

\subsection{Notation for $GL(V)$-equivariant degenerations}
Our relations among $GL(V)$-equivariant orbit classes will be given by
degenerating orbits.  We now provide the basic setup for our
degenerations.

\begin{def}
  \label{defleadsto} Let
  $m \cdot Z, m_{1} \cdot Z^{1}, \dots, m_{k} \cdot Z^{k} \subset
  \on{Sym}^dV^{\vee}$ be cycles (i.e. irreducible, closed subvarieties
  with attached positive multiplicities $m, m_{i} \in \mathbb{Z}$),
  each preserved by the action of $GL(V)$.  Then we write
  $$Z \leadsto \sum_{i=1}^{k}Z^{i}$$ if there exists:
  \begin{enumerate}
  \item An open neighborhood $U \subset \mathbb{A}^{1}$ containing the
    point $0$,
  \item a closed subvariety $W \subset U \times \on{Sym}^dV^{\vee}$
    flat over $U$ and invariant under the action of $GL(V)$ (acting on
    the second factor) with the property that

  \item there exists a point $u \in U$ such that the fiber
    $W_{u} \subset \{u\} \times \on{Sym}^dV^{\vee}$ is equal to $Z$
    with multiplicity $m$, and the fiber
    $W_{0} \subset \{0\} \times \on{Sym}^dV^{\vee}$ (the
    scheme-theoretic fiber of $W \to U$ over $0$) is the union of the
    $Z^{i}$'s with $Z^{i}$ having multiplicity $m_{i}$.
  \end{enumerate}
\end{def}

(The concept of $\leadsto$ affords its obvious projectivized version
for $GL(V)$-equivariant cycles in $\P \on{Sym}^{d}V^{\vee}$.)

We will primarily be interested in the notion of $\leadsto$ in the
context were $Z$ is an orbit closure $\Orb(F)$. This is the subject of
the next section.


\subsection{Families of orbits}
\label{degenerationsection}
Our intention here is to gather the basic degeneration tools specific
to orbit closures, tools we will repeatedly use throughout the
paper. Every degeneration in the paper will implicitly use the
framework described below.

Let $(U,0)$ denote an open neighborhood of $0$ in $\mathbb{A}^{1}$ ,
and suppose
\begin{align*}
  \alpha: U \to \P \on{Sym}^{d}V^{\vee}
\end{align*}
is a map, inducing a corresponding family of degree $d$ hypersurfaces
\[\pi: \mathcal{X} \to U.\]

Next, let
\[\mathring{\mathcal{Y}} \subset U \times PGL(V) \times \P
  \on{Sym}^{d}V^{\vee}\] denote the closed subset consisting of
triples $(u,g,X)$ satisfying $X = g \cdot \alpha(u)$.  In other words,
$\mathring{\mathcal{Y}}$ is the graph of the action map
$(u,g) \mapsto g \cdot \alpha(u)$.
\begin{defn}
  \label{def:Y}
  Define
\begin{align}
  \label{eq:Y}
  \mathcal{Y} \subset U \times \P \on{End}(V) \times \P \on{Sym}^{d}V^{\vee}
\end{align}
to be the closure of $\mathring{\mathcal{Y}}$ (using the natural open
inclusion $PGL(V) \subset \P \on{End}(V)$).
\end{defn}
Since $PGL(V) \subset \P \on{End}(V)$ is open, and since graphs are
closed, it follows that $\mathring{\mathcal{Y}} \subset \mathcal{Y}$
is an open set.  The map sending $(u,g) \in U \times PGL(V)$ to
$(u,g,g \cdot \alpha(u)) \in \mathring{\mathcal{Y}}$ clearly induces
an isomorphism which we denote by
\[\iota: U \times PGL(V) \to \mathring{\mathcal{Y}}\]

We let $\pi_{1}, \pi_{2}, \pi_{3}$ denote projections of $\mathcal{Y}$
to the respective factors
$U, \P \on{End}(V), \P \on{Sym}^{d}V^{\vee}$. For each $u \in U$, we
let $\mathcal{Y}_{u}$ denote the scheme-theoretic fiber
$\pi_{1}^{-1}(u)$. Observe that, for a general point $u \in U$,
$\mathcal{Y}_{u} := \pi_{1}^{-1}(u) \subset \mathcal{Y}$ is
irreducible, and therefore $\pi_{3}(\pi_{1}^{-1}(u))$ is the orbit
closure $\P\Orb(\mathcal{X}_{u})$.

Since $\pi_{1}: \mathcal{Y} \to U$ is flat ($\mathcal{Y}$ is
irreducible and $U$ is a smooth curve), every fiber of $\pi_{1}$ has
pure dimension equal to $\dim PGL(V)$.  The geometry of the special
fiber $\mathcal{Y}_{0}$ will be of utmost importance to us.

We let
\[\mathcal{A}_{0} \subset \mathcal{Y}_{0}\] denote the irreducible
component $\overline{\iota(\{0\} \times PGL(V))}$.  Then, by
construction, the map
$\pi_{3}: \mathcal{A}_{0} \to \P\on{Sym}^{d}V^{\vee}$ restricts on
$\{0\} \times PGL(V)$ to the action map $g \mapsto g \cdot \alpha(0)$,
and hence \[\pi_{3}(\mathcal{A}_{0}) = \P \Orb(\mathcal{X}_{0}).\]
Furthermore, in terms of cycles,
\[\pi_{3*}[\mathcal{A}_{0}] = \# \on{Aut}(\mathcal{X}_{0}) \cdot [\P \Orb(\mathcal{X}_{0})].\]
Along with $\mathcal{A}_{0}$, there may be several other components of
the fiber $\mathcal{Y}_{0} := \pi_{1}^{-1}(0)$, and our aim is to gain
a better understanding of these.

To that end, we must introduce the idea of {\sl twisting by
  $1$-parameter families.}  Let $U^{\times} = U \setminus \{0\}$, and
suppose we have a morphism
\begin{align}
  \label{eq:gammatwist}
  \gamma: U^{\times} \to PGL(V);
\end{align}
any such $\gamma$ will be called a {\sl $1$-parameter family.}  We
will abuse notation and let $\gamma$ also denote its canonical
extension $U \to \P \on{End}(V)$ across $0 \in U$.  In the presence of
a $1$-parameter family $\gamma$, we obtain a new family of
hypersurfaces,
\begin{align}
  \label{eq:twist}
  \alpha^{\gamma}: U^{\times} \to \P \on{Sym}^{d}V^{\vee} 
\end{align}
defined by $\alpha^{\gamma}(u) := \gamma(u) \cdot \alpha(u)$. We call
$\alpha^{\gamma}$ the {\sl twist} of $\alpha$ by $\gamma$.  Again, we
will abuse notation and use $\alpha^{\gamma}$ to also denote the
natural extension to $U$. Then, $\alpha^{\gamma}$ induces a new family
of hypersurfaces which we denote by
$\pi^{\gamma}: \mathcal{X}^{\gamma} \to U$.

Now consider the map
\[\iota^{\gamma}: U \times PGL(V) \to \mathcal{Y}\] defined by the formula
\[\iota^{\gamma}(u,g) = (u, g \circ \gamma(u), g \cdot
  \alpha^{\gamma}(u)).\]

Note that the appropriate restriction of $\iota^{\gamma}$ induces an
isomorphism between $U^{\times} \times PGL(V)$ and
$\mathring{\mathcal{Y}}^{\times}$, the latter being the open set
$\pi^{-1}(U^{\times}) \cap \mathring{\mathcal{Y}}$.  Denote by
$\iota^{\gamma}_{0}: PGL(V) \to \mathcal{Y}_{0}$ for the restriction
of $\iota^{\gamma}$ to the subset $\{0 \} \times PGL(V)$; simply put,
\[\iota^{\gamma}_{0}(g) = (0, g \circ \gamma(0), g \cdot
  \alpha^{\gamma}(0)).\]

Under the action of $PGL(V)$ (post-composition on middle factor) on
$\mathcal{Y}$, $\iota^{\gamma}_0$ is the action map for the point
$(0, \gamma(0), \alpha^{\gamma}(0))$.

\begin{defn}
  \label{def:fullgamma}
  We say the $1$-parameter family $\gamma$ is {\sl full for $\alpha$}
  if the $PGL(V)$ orbit of the point
  \[(0, \gamma(0), \alpha^{\gamma}(0)) \in \mathcal{Y}\] is full
  dimensional. Equivalently, $\gamma$ is full for $\alpha$ if
  $\iota^{\gamma}_{0}$ is quasi-finite onto its image.
\end{defn}

\begin{defn}
  \label{def:Agamma}
  Suppose $\gamma$ is full for $\alpha$. Define
  \[\mathcal{A}_{\gamma} \subset \mathcal{Y}_{0} \] to be the closure
  of the orbit of $(0, \gamma(0), \alpha^{\gamma}(0))$.  
\end{defn}

By fullness of $\gamma$, $\mathcal{A}_{\gamma}$ is an irreducible
component of $\mathcal{Y}_{0}$ which contains a dense
$PGL(V)$-orbit. The next lemma will be used to help us distinguish
among irreducible components of $\mathcal{Y}_{0}$.

\begin{lem}
  \label{gamma12}
  Suppose that $\gamma_{1,2}: U^{\times} \to PGL(V)$ are two
  $1$-parameter families which are full for $\alpha$, and let
  $\mathcal{A}_{\gamma_{1}}$ (resp. $\mathcal{A}_{\gamma_{2}}$) denote
  the corresponding irreducible component of $\mathcal{Y}_{0}$ as in
  \Cref{def:Agamma}.  Further suppose that
  $\gamma_{i}(0) \in \P \on{End}(V)$ and
  $\gamma_{2}(0) \in \P \on{End(V)}$ are not translates under the
  action of $PGL(V)$ on $\P \on{End}(V)$ via post-composition.

  Then
  \[\mathcal{A}_{\gamma_{1}} \neq \mathcal{A}_{\gamma_{2}}.\]
\end{lem}

\begin{proof}
  We will argue the contrapositive.

  Suppose
  $\mathcal{A}_{\gamma_{1}} = \mathcal{A}_{\gamma_{2}} = \mathcal{A}$.
  Then both points $(0,\gamma_{i}(0), \alpha^{\gamma_{i}}(0))$ are in
  $\mathcal{A}$, and both have full orbits.  As $\mathcal{A}$ is
  irreducible and has a dense orbit, it follows that both of the above
  points are in the same $PGL(V)$ orbit.

  Since $\pi_{2}: \mathcal{Y}_{0} \to \P \on{End}(V)$ is
  $PGL(V)$-equivariant, it follows that
  $ \gamma_{1}(0) = \gamma_{2}(0)$ are also in the same $PGL(V)$-orbit
  in $\P \on{End}(V)$, which is what we needed to show.
\end{proof}

\begin{rmk}
  \label{imageskernels}
  Observe that the condition ``$\gamma_{1}(0)$ and $\gamma_{2}(0)$ are
  not $PGL(V)$ translates'' is equivalent to the condition that the
  projective endomorphisms $\gamma_{i}(0)$ do not have the same {\sl
    kernel} subspaces in $\P V$ in the case of the action by {\sl
    post}-composition.  We will use this observation repeatedly.
\end{rmk}

\begin{prin}
  \label{principle}
  Suppose $\gamma_{i}$, $1\leq i\leq n$ are $1$-parameter families,
  full for $\alpha$, and suppose each pair $\gamma_{i},\gamma_{j}$
  satisfy the hypothesis in \Cref{gamma12}.  Then, the equivariant
  class
  \[\p_{\mathcal{X}_{u}}-\sum_{i=1}^{n}\p_{\mathcal{X}^{\gamma_{i}}_{0}}\]
  can be represented by a nonnegative sum of equivariant fundamental
  classes of effective cycles.

  If, in addition, the predegrees of the hypersurfaces
  $\mathcal{X}^{\gamma_{1}}_{0},\ldots,\mathcal{X}^{\gamma_{n}}_0$ add
  up to the predegree of $\mathcal{X}_{u}$, then,

  \[\p_{\mathcal{X}_u}=\sum_{i=1}^{n}\p_{\mathcal{X}^{\gamma_{i}}_{0}}.\]
\end{prin}

\begin{proof}
  Recall the variety $\mathcal{Y}$ from \Cref{def:Y}, and
  $\overline{\mathcal{Y}} \subset U \times \P \on{Sym}^{d}V^{\vee}$
  its image under projection onto first and third factors
  $(\pi_{1},\pi_{3})$. In light of the hypotheses in \Cref{gamma12},
  we obtain pairwise distinct irreducible components
  $\mathcal{A}_{i} \subset \mathcal{Y}_{0}$ of $\mathcal{Y}_{0}$,
  $i=1, \dots, n$. By their constructions, the composite
  \[\pi_{3}\circ \iota^{\gamma}_{0}: PGL(V) \to \P \on{Sym}^{d}V^{\vee}\] the action map
  corresponding to each hypersurface $\mathcal{X}^{\gamma_{i}}_{0}$.

  Therefore, in the cycle $\pi_{3*}[\mathcal{Y}_{0}]$, the coefficient
  of $[\P \Orb(\mathcal{X}^{\gamma_{i}}_{0})]$ is
  $\deg(PGL(V) \to\P \Orb(\mathcal{X}^{\gamma_{i}}_{0}))$, i.e.
  $\# \on{Aut}(\mathcal{X}_{0}^{\gamma_{i}})$.  Since
  $\pi_{3*}[\mathcal{Y}_{u}] = \# \on{Aut}(\mathcal{X}_{u})[\P
  \Orb(\mathcal{X}_{u})]$, we deduce that the cycle
  \[\# \on{Aut}(\mathcal{X}_{u})[\P \Orb(\mathcal{X}_{u})] -
    \sum_{i=1}^{n} \# \on{Aut}(\mathcal{X}^{\gamma_{i}}_{0})[\P
    \Orb(\mathcal{X}^{\gamma_{i}}_{0})]\] has an effective
  representative supported on the irreducible components of
  $\pi_{3}(\overline{\mathcal{Y}}_{0})$ not accounted for by the
  orbits of the hypersurfaces $\mathcal{X}^{\gamma_{i}}$. The
  effectivity statement of
  \[\p_{\mathcal{X}_{u}}-\sum_{i=1}^{n}\p_{\mathcal{X}^{\gamma_{i}}_{0}}\]
  now follows.

  Finally, assume that the predegrees of
  $\mathcal{X}^{\gamma_{i}}_{0}$ add to the predegree of
  $\mathcal{X}_{u}$. Let $\Lambda \subset \P \on{Sym}^{d}V^{\vee}$ be
  a general linear space of codimension $\dim PGL(V)$ intersecting
  $\P \Orb(\mathcal{X}_{u})$ and
  $\P \Orb(\mathcal{X}^{\gamma_{i}}_{0})$ transversely for all $i$.
  Then, since $\deg \pi_{3}^{-1}(\Lambda)$ is the sum of predegrees of
  the $\mathcal{X}^{\gamma_{i}}_{0}$, and since this degree is
  constant over $U$, it follows that there are no irreducible
  components $\mathcal{B} \subset \mathcal{Y}_{0}$ other than the
  $\mathcal{A}_{i}$ having the property that $\pi_{3}(\mathcal{B})$ is
  an irreducible component of
  $\pi_{3}(\mathcal{Y})$. Thus,
  \[(\pi_{1},\pi_{3})(\mathcal{Y}) \subset U \times \P
    \on{Sym}^{d}V^{\vee}\] induces
  \[\# \on{Aut}(\mathcal{X}_{u})[\P \Orb(\mathcal{X}_{u})] \leadsto
    \sum_{i}\# \on{Aut}(\mathcal{X}^{\gamma_{i}}_{0})[\P
    \Orb(\mathcal{X}^{\gamma_{i}}_{0})],\] and the equality
  \[\p_{\mathcal{X}_u}=\sum_{i=1}^{n}\p_{\mathcal{X}^{\gamma_{i}}_{0}}\]
  follows.
\end{proof}

For many of our applications, \Cref{principle} will suffice. But in
\Cref{nodedegeneration}, we will find ourselves in a situation that
doesn't quite obey the hypotheses of \Cref{principle}, yet the
conclusion will still follow. Our adjustment is to replace the
compactification $PGL(V) \subset \P \on{End}(V)$ by another variety.
\begin{defn}
\label{def:Inv}
Define
\begin{align*}
  \label{Inv}
  \mathsf{Inv}(V) \subset \P \on{End}(V) \times \P \on{End}(V)
\end{align*}  to be 
the closure of the graph of the inversion morphism
$i: PGL(V) \to PGL(V)$.  In other words, $\mathsf{Inv}(V)$ is the
closure of the set of pairs $(g,g^{-1})$, where $g \in PGL(V)$.
\end{defn}

The group $PGL(V)$ acts on $\mathsf{Inv}(V)$ by {\sl post}-composition
on the first factor and {\sl pre}-composition on the second factor,
i.e. $h \in PGL(V)$ acts on a pair $(e_1, e_2) \in \mathsf{Inv}(V)$ to
give $(h \circ e_{1}, e_{2} \circ h^{-1})$.

The varieties $\mathcal{Y}$, $\mathcal{A}_{\gamma}$, and
\Cref{gamma12} have their immediate generalizations to the
$\mathsf{Inv}(V)$ setting, and we will abuse notation and use the same
letters whenever we are in this analogous setting.

\begin{prin}
\label{principle2}
Let $\gamma_1,\ldots,\gamma_n$ be $1$-parameter families
$\gamma_i: U^{\times}\to PGL(V)$ which are full for $\alpha$ and
continue to let $\gamma_{i}: U \to \mathsf{Inv}(V)$ denote the unique
extension over $U$.  Suppose furthermore that:
\begin{itemize}
\item The points $\gamma_i(0) \in \mathsf{Inv}(V)$ are in distinct
  $PGL(V)$-orbits of $\mathsf{Inv}(V)$.
\end{itemize}
Then, the equivariant class
$\p_{\mathcal{X}_{u}}-\sum_{i=1}^{n}\p_{\mathcal{X}^{\gamma_{i}}_{0}}$
can be represented by a nonnegative sum of effective cycles. If
additionally the predegrees of $\mathcal{X}^{\gamma_{i}}_{0}$,
$i=1, \ldots, n$ sum up to the predegree of $\mathcal{X}_{u}$, then
\[\p_{\mathcal{X}_{u}}=\sum_{i=1}^{n}\p_{\mathcal{X}^{\gamma_{i}}_{0}}.\]
\end{prin}

We omit the proof, as it is similar to the proof of \cref{principle}.

\subsection{Specialize to $r=2$}
\label{sec:specialize2}

From here onward, unless specified otherwise, we assume $V$ is a
$3$-dimensional vector space.  We will use the letter $C$ with
subscripts and superscripts to indicate a plane curve.  Apart from the
specific plane curves with particular subscripts (e.g. ``$C_{BN}$''),
the simple letter $C$ may represent different curves in different
sections.  We hope its meaning will be clear in context.

\section{Known orbit classes of special quartic curves}
\label{stratasection}

Our intention in the section is to get a few critical orbit classes
$\p_{C}$ of curves in our hands without using degeneration. We use
Kazarian's formulas for counting curves with $A_6, D_6$ and $E_6$
singularities. It is known that in the space of quartic curves, the
set of curves with such singularities form three respective full
(i.e. $8$-) dimensional orbits (\Cref{oneorbit}). Kazarian's formulas
then directly yield the equivariant orbit classes of these three
orbits (\Cref{A6E6orbit}). We also record the computation of $\p_{Q}$
where $Q \subset \P \on{Sym}^{4}V^{\vee}$ is a complete quadrilateral,
i.e. the union of four general lines.

We begin with the following calculation of Kazarian \cite[Theorem
1]{Kazarian}:
\begin{prop}
  \label{Kaztheman}
  Let $\mc{S}\to B$ be a smooth morphism of varieties whose fibers are
  smooth surfaces. Let $L$ be a line bundle on $\mc{S}$ and $\sigma$
  be a section of $L$ cutting out a family of curves
  $\mc{C}\subset \mc{S}$. The virtual classes $[Z_{A_6}]$
  (respectively $[Z_{D_6}]$ and $[Z_{E_6}]$) supported on points
  $p\in \mc{S}$ where the fiber of $\mc{C}\to B$ has an $A_6$
  (respectively $D_6$ and $E_6$) singularity at $p$ are given by:
\begin{align*} 
[Z_{A_6}]=&\ u (-c_1 + u) (c_2 - c_1 u + u^2) (720 c_1^4
  - 1248 c_1^2 c_2 + 156 c_2^2 - 1500 c_1^3 u\\& + 1514 c_1 c_2 u+
  1236 c_1^2 u^2 - 485 c_2 u^2 -
  487 c_1 u^3 + 79 u^4)\\
  [Z_{D_6}] =&\ 2 u (-c_1 + u) (4 c_2 - 2 c_1 u + u^2) (c_2 - c_1 u + u^2) (12 c_1^2 - 
   6 c_2 - 13 c_1 u + 4 u^2)\\
  [Z_{E_6}] =&\ 3 u (-c_1 + u) (2 c_1^2 + c_2 - 3 c_1 u + u^2)(4 c_2 - 2
              c_1 u + u^2) (c_2 - c_1 u + u^2)
   \end{align*}
 where $c_i:=c_i(T_{\mc{S}/B})$ and $u=c_1(L)$. 
\end{prop}


\begin{prop}
\label{oneorbit}
The set of irreducible quartic plane curves with an $A_6$
(respectively $D_6$ and $E_6$) singularity having full dimensional
orbit consists of a single $PGL(V)$ orbit.
\end{prop}
\begin{proof}
  The case of $D_6$ singularities is clear, since one of the branches
  of the singularity must be a line, and therefore $D_6$ is the union
  of a nodal cubic along with one of the branch lines at the node.
  Hence such a curve must be the union of a nodal cubic with a tangent
  branch line, constituting a single orbit.

  The fact that irreducible plane quartics with an $A_6$ or $E_6$
  singularity form an irreducible subvariety of codimension $6$ in the
  projective space $\mb{P} \on{Sym}^{4}V^{\vee}$ of all quartics
  follows from explicit classification, for example \cite[Section
  3.4]{Nejad2}. 

  That an orbit of a general curve with such a singularity is
  $8$-dimensional can be checked by the formulas for their pre-degrees
  as found in \cite[Examples 5.2 and 5.4]{AF00}, which gives a nonzero
  result. This proves the proposition.
\end{proof}

\begin{defn}
  \label{def:specialcurves}
  Let $C_{A_6}$ and $C_{E_6}$ denote the rational quartic curves with
  an $A_6$ and $E_6$ singularity, respectively, which have full
  dimensional orbits. By \Cref{oneorbit}, this definition is
  well-defined up to projective equivalence.
\end{defn}

There are explicit equations for $C_{A_6}$ and $C_{E_6}$ (see for
example \cite[Section 3.4]{Nejad2}):
\begin{align}
  \label{eq:sings}
  C_{A_6}: (X^2+YZ)^2+2YZ^3=0\\
  C_{D_6}: Z(ZXY+X^3+Z^3)=0\\
  C_{E_6}: Y^3Z+X^4+X^2Y^2=0.
\end{align}

\begin{cor}
\label{A6E6orbit}
We have
\begin{align*}
  \p_{C_{A_6}}&= 3\cdot112 (9 c_1^3 + 12 c_1 c_2 - 11 c_3) (2 c_1^3 + c_1 c_2 + c_3)\\
  \p_{C_{D_6}}&=3\cdot 64 (18 c_1^6 + 33 c_1^4 c_2 + 12 c_1^2 c_2^2 - 85 c_1^3 c_3 - 11 c_1 c_2 c_3 - 
                7 c_3^2)\\
  \p_{C_{E_6}}&= 2\cdot 48 (2 c_1^3 + c_1 c_2 + c_3) (9 c_1^3 - 6 c_1 c_2 + 7 c_3),
\end{align*}
where $\#\on{Aut}(C_{A_6})=\# \on{Aut}(C_{D_6})=3$ and
$\#\on{Aut}(C_{E_6})=2$.
\end{cor}

We will also verify the result for $\p_{C_{D_6}}$ independently in
\Cref{CAN}.

\begin{proof}
  We apply \Cref{Kaztheman} to the case where $B$ is an arbitrary base
  variety and $\mc{V} \to B$ is an arbitrary rank $3$ sub-bundle. Let
  $T$ be the relative tangent bundle of $\mb{P}(\mc{V})\to B$. By the
  splitting principle and the relative Euler exact sequence for
  projective bundles, we get:
\begin{align*}
  c_1(T)&= c_1(\mc{V}) + 3c_1(\ms{O}_{\mb{P}(\mc{V})}(1)) \\
  c_2(T) &= c_2(\mc{V}) + 2 c_1(\mc{V}) c_1(\ms{O}_{\mb{P}(\mc{V})}(1)) + 3 c_1(\ms{O}_{\mb{P}(\mc{V})}(1))^2.
\end{align*}
Now, we substitute $u=4c_1(\ms{O}_{\mb{P}(\mc{V})}(1))$ in the
formulas for $[Z_{A_6}]$, $[Z_{D_6}]$ and $[Z_{E_6}]$ \Cref{Kaztheman}
and apply push-forward along the projection
$\mb{P}(\mc{V})\xrightarrow{\pi} B$. This yields
\begin{align*}
  \pi_{*}[Z_{A_6}]=&\ 112 (9 c_1(\mc{V})^3 + 12 c_1(\mc{V}) c_2(\mc{V}) - 11 c_3(\mc{V})) (2 c_1(\mc{V})^3 + c_1(\mc{V}) c_2(\mc{V}) + c_3(\mc{V}))\\
  \pi_{*}[Z_{D_6}]=&\ 64 (18 c_1(\mc{V})^6 + 33 c_1(\mc{V})^4 c_2(\mc{V}) + 12 c_1(\mc{V})^2 c_2(\mc{V})^2 - 85 c_1(\mc{V})^3 c_3(\mc{V}) \\
  & - 11 c_1(\mc{V}) c_2(\mc{V}) c_3(\mc{V}) - 7 c_3(\mc{V})^2)\\
  \pi_{*}[Z_{E_6}]=&\ 48 (2 c_1(\mc{V})^3 + c_1(\mc{V}) c_2(\mc{V}) + c_3(\mc{V})) (9 c_1(\mc{V})^3 - 6 c_1(\mc{V}) c_2(\mc{V}) + 7 c_3(\mc{V})).
\end{align*}
Now, $\pi_{*}[Z_{A_6}]$, $\pi_{*}[Z_{D_6}]$, and $\pi_{*}[Z_{E_6}]$
respectively give the formulas for $[\Orb({C_{A_6}})]_{GL(V)}$,
$[\Orb({C_{D_6}})]_{GL(V)}$, $[\Orb({C_{E_6}})]_{GL(V)}$, as they are
also the result of pulling back the relative cycles
$(\Orb({C_{A_6}}))_{\mc{V}}$, $(\Orb({C_{D_6}}))_{\mc{V}}$, and
$(\Orb({C_{E_6}}))_{\mc{V}}$ under a generic section
$B\to \on{Sym}^4\mc{V}^{\vee}$.

The statement on the automorphisms of $C_{A_6}$ and $C_{E_6}$ come
from a direct analysis of equations. Alternatively, one could compare
the predegrees of $C_{A_6}$ and $C_{E_6}$ with the projective versions
of $[Z_{A_6}]$ and $[Z_{E_6}]$ using \cite[Examples 5.2 and 5.4]{AF00}
and \Cref{projectivethom}.
\end{proof}

In order to calculate the orbit class of a general quartic with a
triple point, we will need to know $\p_Q$ in the case where $Q$ is the
union of four lines with no three concurrent, i.e. a complete
quadrilateral.  The method is simply to ``present'' the orbit closure
$\P \Orb(Q)$ by a more accessible variety.

\begin{prop}
\label{fourlines}
Let $Q$ be the union of four lines, no three concurrent. Then,
\begin{align*}
  \p_Q = 24\cdot 16 (18 c_1^6 + 33 c_1^4 c_2 + 12 c_1^2 c_2^2 + 131 c_1^3 c_3 + 
  153 c_1 c_2 c_3 - 147 c_3^2).
\end{align*}
Here, $\#\on{Aut}(Q)=24$.
\end{prop}

\begin{proof}
  We will closely follow the ideas in \cite[Theorem 3.1]{FNR06}. Let
  $\mathcal{V} \to B$ be an arbitrary rank $3$ vector bundle. Consider
  the map
  $$\phi: \mb{P}(\mc{V}^{\vee})^4\to \mb{P}(\on{Sym}^4\mc{V}^{\vee}),$$
  which, fiber by fiber over $B$, restricts to the map induced by
  sending a tuple of linear forms $(L_1, \dots, L_{4})$ to the quartic
  form $L_1 \cdot L_2, \cdot L_{3} \cdot L_{4}$ . Then, $\phi$ maps
  $4!$ to 1 onto $\mb{P}\Orb(Q)_{\mathcal{V}}$, so
  $[\mb{P}\Orb(Q)]_{\mathcal{V}}=\frac{1}{24}\phi_{*}(1)$.

  Let $H=c_{1}\ms{O}_{\mb{P}\on{Sym}^4\mc{V}^{\vee}}(1)$ and set
\begin{align*}
  \alpha := H^{14}+c_1(\on{Sym}^4\mc{V}^{\vee})H^{13}+\cdots+c_{14}(\on{Sym}^4\mc{V}^{\vee}). 
\end{align*}
The Leray relation states that
$\alpha H + c_{15}(\on{Sym}^4\mc{V}^{\vee})=0$, and it follows from
this that the integral
\begin{align*}
  \int_{\mb{P}(\on{Sym}^4\mc{V}^{\vee})\to B}\alpha \cdot \beta
\end{align*}
returns the ``constant term'' (with respect to $H$) of any class
$\beta$. By this, we mean that any class
$\beta\in A^{\bullet}(\mb{P}(\on{Sym}^4\mc{V}^{\vee})$ can be written
as a polynomial in $H$ of degree at most $14$, with coefficients being
pullbacks of classes of $A^{\bullet}(B)$ and that integrating against
$\alpha$ and pushing forward to $B$ extracts the $H^{0}$ or constant
term of $\beta$.  (We use the notation $\int_{X \to Y}$ to denote
pushforward of classes along a map $X \to Y$.)

To finish, we choose $\beta :=\frac{1}{24}\phi_{*}(1)$ and apply the
push-pull formula to reduce our problem to the evaluation of
\begin{align*}
    \frac{1}{24}\int_{\mb{P}(\mc{V}^{\vee})^4\to B}\phi^{*}(\alpha).
\end{align*}
This evaluation is now standard (and we leave it to the reader), given
that $\phi^{*}H = h_1 + h_2 + h_3 + h_4$, where $h_{i}$ is the
pullback of the relative hyperplane class under the projection
$p_{i}: \P(\mc{V}^{\vee})^{4} \to \P(\mc{V}^{\vee})$ onto the $i$-th
factor.  The end result yields $\p_Q$ as stated in the proposition, in
light of \Cref{projectivethom}.
\end{proof}


\section{Degeneration I: Splitting off a line}

In this section, we analyze our first degeneration: We investigate how
the orbit closure specializes as a degree $d$ smooth curve specializes
to a general degree $e$ smooth curve together with $d-e$ general
lines. Let $U = \mb{A}^{1}$, with coordinate $t$ vanishing at $0$. We
will follow the terminology and framework of
\Cref{degenerationsection}.

\subsection{Framework}
\label{sec:frameworksplit}

Let $F(X,Y,Z)$ and $G(X,Y,Z)$ be forms of degrees $d-1$ and $d$
respectively, and assume $G$ does not vanish identically on the line
$\{X=0\}$. Let \[\alpha: U \to \P \on{Sym}^{d}V^{\vee}\] be given by
the formula $t \mapsto XF + tG$. Finally, let
\[\gamma: U^{\times} \to PGL(V)\] be the $1$-parameter family of
matrices
\[\begin{pmatrix}
    t^{-1} & 0 & 0 \\
    0 & 1 & 0\\
    0 & 0 & 1\\
  \end{pmatrix}.\] 

We maintain this framework throughout the rest of the section.

\begin{lem}
\label{line}
In the setting of \Cref{sec:frameworksplit}, the curve
$\alpha^{\gamma}(0) \in \P \on{Sym}^{d}V^{\vee}$ has equation
$X \cdot F(0,Y,Z)+G(0,Y,Z) = 0$.

\end{lem}

\begin{proof}
 Unraveling the definition of the family of curves
 $\alpha^{\gamma}: U \to \P \on{Sym}^{d}V^{\vee}$, it suffices to
 prove
 \begin{align*}
   \lim_{t\to 0}t^{-1}((tX)F(tX,Y,Z)+tG(tX,Y,Z))&=X \cdot F(0,Y,Z)+G(0,Y,Z),
\end{align*}
which is immediate.
\end{proof}

Let us interpret \Cref{line} in geometric terms, under the further
condition that $F$ and $G$ are suitably general.  The original family
$\alpha$ represents a general degree $d$ curve $G = 0$ specializing to
a reducible curve containing a line, $X \cdot F = 0$.  Upon twisting
by $\gamma$, the new family $\alpha^{\gamma}$ now specializes the same
general curve to the curve $X \cdot F(0,Y,Z) + G(0,Y,Z)$, which is a
{\sl general} curve among those possessing a multiplicity $(d-1)$
singular point.

\begin{rmk}
  \label{rmk:split}
  Notice that the limiting endomorphism $\gamma(0) \in \P \on{End}(V)$
  is given by the matrix
\[\begin{pmatrix}
    1 & 0 & 0 \\
    0 & 0 & 0\\
    0 & 0 & 0\\
  \end{pmatrix}\] and therefore the kernel space of $\gamma(0)$ is the
line $X=0$ which was a component of the curve $\alpha(0)$.
\end{rmk}

We can slightly generalize the analysis in \Cref{line} to the
situation where a general degree $d$ curve $C$ specializes to a curve
$D$ (via a map $\beta: U \to \P \on{Sym}^{d}V^{\vee}$), where $D$ is
the union of a general degree $e \geq 1$ curve along with $d-e$
generally chosen lines $L_1, \dots, L_{d-e}$.  By adapting the family
of matrices $\gamma$ to this more general situation, we obtain
$\gamma_{i}: U^{\times} \to PGL(V)$, one per each line $L_{i}$,
satisfying:
\begin{enumerate}
\item $\beta^{\gamma_{i}}(0)$ is a degree $d$ curve general among
  those with a multiplicity $d-1$ point, and
\item the kernel of the endomorphism $\gamma_{i}$ is the line
  $L_{i} \subset \P V$. (see \Cref{rmk:split}.)
\end{enumerate}

\begin{prop}
\label{linespecialization}
Assume $d\geq 4$ and let $C$ be a general degree $d$ curve, $C_{d-1}$
a general degree $d$ curve possessing a point of multiplicity $d-1$,
and $D$ the union of a general degree $e \geq 0$ curve together with
$d-e$ general lines. Then,
\begin{align*}
  \p_{C} =  (d-e)\p_{C_{d-1}}+ \p_{D}.
\end{align*}
\end{prop}

\begin{proof}
  This will be a direct application of \Cref{principle}.  Let
  $\beta: U \to \P \on{Sym}^{d}V^{\vee}$ and $\gamma_{i}$
  $i=1, \dots, d-e$ be as in the generalized setup immediately prior
  to the statement of the proposition.

  The hypotheses of \Cref{principle} are met:

\begin{enumerate}
\item $C_{d-1}$ has full dimensional orbit, as does $D$, (thanks to
  the assumption $d \geq 4$).
\item The kernels of $\gamma_{i}(0)$ are
  distinct, namely the lines $L_{i}$.

\item It remains to check the compatibility of predegrees.  For this,
  (by further specializing $D$ to a union of $d$ general lines) it
  suffices to consider the case $e=0$, where now the task remaining is
  to check that the predegree of a general degree $d$ curve $C$ is $d$
  times the predegree of $C_{d-1}$, a general degree $d$ curve with a
  point of multiplicity $d-1$ plus the predegree of the union of $d$
  general lines. Finally, this last check is accomplished by plugging
  into the formulas in \cite[Examples 3.1, 4.2]{AF00} and \cite{AF93}.
\end{enumerate}
The proposition follows from \Cref{principle}.
\end{proof}

\subsection{Summary}
The proof of \Cref{linespecialization} provides a comprehensive
understanding of the $t \to 0$ flat limit of orbit closures
$\P \Orb(C_{t})$ , $t \in U$ if $C_{t}$ is a family of general curves
specializing to a curve $D$ general among those which contain $d-e$
lines as components. Apart from $\P \Orb(D)$, we find $d-e$ other
orbits $\P \Orb(C_{i})$, where $C_{i}$ are general among curves
possessing a $d-1$-fold point.


\section{Degeneration II: Acquiring nodes and cusps}
\label{sec:nodecusp}

In this section, we establish the effect of acquiring an ordinary node
or cusp (a cusp is a singularity with analytic equation $y^{2}=x^{3}$)
on the polynomial $\p_{C}$, in the case of arbitrary $d \geq 4$.  A
node singularity $p$ of a plane curve $C$ is called {\sl ordinary} if
both tangent lines intersect $C$ with multiplicity $3$ at
$p$. Similarly, we call a cusp singularity ordinary if no line meets
it with multiplicity $\geq 4$.  Throughout, $U$ will denote an
appropriate open neighborhood of $0$ in $\mb{A}^{1}$ and $t$ will
denote a coordinate around $0$. We let $R$ denote the coordinate ring
of $U$, and we let $v: R \to \mb{N}$ denote the valuation
corresponding to the point $0 \in U$.

\subsection{Framework and summary of main results}
\label{sec:frameworknodecusp}

We assume \[\alpha: U \to \P \on{Sym}^{d}V^{\vee}\] induces a family
of curves $\pi: \mathcal{C} \to U$ with the following properties:
  \begin{enumerate}
  \item The curve $C_{u} := \pi^{-1}(u)$, $u\in U$ general, is a
    smooth curve with no hyperflexes, and
  \item the curve $C_{0} := \pi^{-1}(0)$ has $\delta$ ordinary nodes
    and $\kappa$ ordinary cusps, and
  \item $\mathcal{C}$ is a smooth surface, and $C_{0}$ has no
    hyperflexes.
  \end{enumerate}

  Certain curves of special significance will arise, so we collect
  their definitions here.

  \begin{defn}
    \label{specialcurves} Define the curves $C_{BN}$, $C_{AN}$,
    $C_{\on{flex}}[j]$, $j \in \k$ as:
    \begin{enumerate}
    \item $C_{BN}: Z^{d-3}(XYZ + X^3 + Z^3) = 0$\\
    \item $C_{AN}: Z^{d-3}(Y^2Z - X^3 + X^2Z) = 0$\\
    \item $C_{\on{flex}}[j]$: This is the union of a smooth cubic
      curve with $j$-invariant $j$ along with one of its flex lines,
      the line taken with multiplicity $(d-3)$.
    \end{enumerate}
  \end{defn}
 
  Although these curves depend on $d$, we have suppressed it from the
  notation, and hope $d$ is clear from context.  In geometric terms,
  $C_{BN}$ is a nodal cubic union a $d-3$ line tangent to one of the
  branches at the node, taken with multiplicity $d-3$. $C_{AN}$ is
  similarly a nodal cubic curve along with one of its three flex lines
  (at a smooth point), the line taken with multiplicity $(d-3)$.

  In fact, it will turn out that $\p_{C_{\on{flex}}[j]}$ is
  independent of $j \in \k$.  This follows from the following
  proposition, which is proven in \Cref{proof:CANCBN} below.

  \begin{prop}
    \label{prop:CANCBN}
    Keep the setting above. For every $j \in \k$,
  \[\# \on{Aut}(C_{\on{flex}}[j]) \cdot [\P \Orb(C_{\on{flex}}[j])]
    \leadsto \#\on{Aut}(C_{AN})[\P \Orb(C_{AN})] +
    \#\on{Aut}(C_{BN})[\P \Orb(C_{BN})]\] and therefore
  \[\p_{C_{\on{flex}}[j]} = \p_{C_{AN}} + 2\p_{C_{BN}}.\]
\end{prop}

The main result of this section, however, is to prove:

\begin{thm}
  \label{nodecusp}
  
  Assume the setting above. Then
\begin{align}
  \p_{C_{u}} = \p_{C_{0}} + 2\delta  \cdot \p_{C_{BN}} + \kappa \cdot   \p_{C_{\on{flex}}[j]}.
\end{align}
\end{thm}

\subsection{Degeneration to a node}
\label{nodedegeneration}

In this subsection, we assume
\[\alpha_{\on{node}}: U \to \P \on{Sym}^{d}V^{\vee}\] is a family of
curves cut out by a degree $d$ form $F(X,Y,Z)$ with coefficients in
$R$ the coordinate ring of $U$, with the properties:

\begin{enumerate}
\item The curve $C_{0} := \alpha_{\on{node}}(0)$ has an ordinary node
  at $[0:0:1] \in \P V$ with branch lines $X=0$ and $Y=0$, and
\item The total space $\mathcal{C}$ of the family of curves is smooth
  at the node of $C_{0}$.  
\end{enumerate}

In what follows, to ease exposition, rather than making a base changes
$s^n = t$, we will abuse notation and work with the fractional powers
$t^{1/n}$. We also extend the valuation $v$ to such fractional
expressions in the obvious way.

\begin{lem}
\label{nodematrix}
Maintain the setup immediately prior, and let $\gamma_{1}(t)$ denote
the family of matrices
\[ \begin{pmatrix}
    t^{-1/3} & 0 & 0\\
    0 & t^{-2/3} & 0\\
    0 & 0 & 1\\
  \end{pmatrix}.\]

Then, the limiting plane curve
$\alpha_{\on{node}}^{\gamma_{1}}(0) \in \P \on{Sym}^{d}V^{\vee}$ has
equation
\begin{align*}
  \lim_{t\to 0}{t^{-1}(F(t^{\frac{1}{3}}X,t^{\frac{2}{3}}Y,Z))}.
\end{align*}
This curve is projectively equivalent to $C_{BN}$.

Similarly, if $\gamma_{2}(t)$ is the family of matrices
\[ \begin{pmatrix}
    t^{-2/3} & 0 & 0\\
    0 & t^{-1/3} & 0\\
    0 & 0 & 1\\
  \end{pmatrix}\] then $\alpha_{\on{node}}^{\gamma_{2}}(0)$ is the
curve defined by
\begin{align*}
    \lim_{t\to 0}{t^{-1}(F(t^{\frac{2}{3}}X,t^{\frac{1}{3}}Y,Z))}
\end{align*}
and is also projectively equivalent to $C_{BN}$.
\end{lem}

\begin{proof}
  It suffices to just do one of the two cases -- we will prove the
  $\gamma_{2}$ case. We know the following:
  \begin{enumerate}
  \item The coefficient $a_{ij}$ of each monomial $X^iY^jZ^{d-i-j}$
    occurring in $F(X,Y,Z)$ is an element of $R$, and therefore have
    non-negative integer valuations.

  \item The assumptions on $C_0$ in the description of
    $\alpha_{\on{node}}$ imply:
    $$v(a_{0,0}),v(a_{1,0}),v(a_{0,1}),v(a_{2,0}),v(a_{0,2})\geq 1.$$

  \item Since the node singularity of $C_0$ is assumed to be ordinary,
    $v(a_{3,0})=v(a_{0,3})=0$.

  \item Since $\mathscr{C}$ is smooth at the node, it follows that
    $v(a_{0,0})=1$.
  \end{enumerate}

  Given these constraints, a direct check now shows that
  $\frac{2}{3}i+\frac{1}{3}j-1+v(a_{i,j})$ is zero if and only if
  $(i,j)\in \{(0,3),(1,1),(0,0)\}$ and is strictly positive otherwise.
  Therefore the result of substituting $t=0$ into the expression in
  the limit produces a nodal cubic with $(d-3)$-fold branch line, as
  claimed.
\end{proof}

\begin{rmk}
  \label{D6CBN}
  In the case $d=4$, any curve $C_{D_{6}}$ (\Cref{def:specialcurves})
  is equal to the curve $C_{BN}$.
\end{rmk}

\begin{rmk}
  \label{rmk:limitendoNode}
  In the context of \Cref{nodematrix}, observe that the limiting
  endomorphism $\gamma_{1}(0)$ has kernel given by the line $Y=0$,
  which is one of the branches of the node in $C_0$ while the limiting
  endomorphism $\gamma_{2}(0)$ has kernel given by the line $X=0$,
  which is the other branch.  Meanwhile, the $t \to 0$ limit of both
  $\gamma_{i}^{-1}$ have images equal to the node point $[0:0:1]$.
\end{rmk}

\subsection{The degree of $\P \Orb(C_{BN})$}
In light of \Cref{nodecusp}, in order to employ the strategy implicit
in \Cref{principle} we will need to compute the degree of the orbit
closure $\P \Orb(C_{BN})$.  In principle, this can be deduced by
applying the algorithm of Aluffi and Faber in \cite{AF00}.  We provide
an independent calculation in this section, as an extra check.

\begin{prop}\label{prop:predegreepolyCBN}
  Let $d \geq 4$. As a function of the degree $d$, the degree of
  $\P \Orb(C_{BN})$ is the quadratic polyonomial
  $24 + 144\cdot (d-3) + 140 \cdot (d-3)^{2}$. The predegree of
  $C_{BN}$ is $3(24 + 144\cdot (d-3) + 140 \cdot (d-3)^{2})$.
\end{prop}
We will prove \Cref{prop:predegreepolyCBN} in pieces below.  Observe
that the calculation of the degree of $\P \Orb(C_{BN})$ implies the
assertion on the predegree thanks to the fact, left to the reader to
check, that the curve $C_{BN}$ has order $3$ automorphism group.

First, we show that such a quadratic expression in $d$ exists in the
first place.

\begin{lem}\label{lem:predegreepolyCBN}
  Let $d \geq 4$. As a function of the degree $d$, the degree of the
  orbit closure $\P \Orb(C_{BN}) \subset \P \on{Sym}^{d}V^{\vee}$ is a
  quadratic polyonomial \[a + b\cdot (d-3) + c \cdot (d-3)^{2}\] with
  $a,b,c\geq 0$.

  Explicitly $a,b,c$ are the answers to the following enumerative
  problems:
  \begin{enumerate}
  \item $a$ is twice the number of singular cubics through $8$ general
    points in $\P V$, i.e. $a = 24$.
  \item $b$ is $\binom{8}{2}$ times the number of nodal cubics through
    $7$ general points in $\P V$ with a nodal branch line containing a
    fixed $8$-th general point.
  \item $c$ is $\binom{8}{2}$ times the number of nodal cubics through
    $6$ general points having a specified line as a branch of the
    node.
  \end{enumerate}

\end{lem}

\begin{proof} Let
  \begin{align}
    \label{Delta}
    \Delta \subset \P \on{Sym}^{3}V^{\vee} \times \P V^{\vee} \times \P V
  \end{align}
  denote the $8$-dimensional variety which is the closure of the set
  of triples $(C,L,p)$ where $C$ is a nodal cubic curve singular at
  the point $p \in \mb{P} V$ and $L \subset \P V$ is a line containing
  $p$ whose intersection multiplicity with $C$ is strictly greater
  than $2$.

  The variety $\Delta$ has three natural projection maps
  $p_{1}, p_{2}, p_{3}$ to the three respective factors of
  $\P \on{Sym}^{3}V^{\vee} \times \P V^{\vee} \times \P V$. Let $H$
  denote the divisor class on $\Delta$ corresponding to
  $p_{1}^{*}\mc{O}(1)$.  Similarly, let $h$ denote the divisor class
  $p_{2}^{*}(\mc{O}(1))$.

  Let
  \[\nu: \mb{P}\on{Sym}^{3}V^{\vee} \times \mb{P}V^{\vee} \to
    \mb{P}\on{Sym}^{d}V^{\vee}\] denote the map which sends a pair
  $(C,L)$ to the degree $d$ curve $C \cup (d-3) \cdot L$. Then the
  composite map
  \[\nu \circ (p_{1},p_{2}) : \Delta \to \P \on{Sym}^{d}V^{\vee}\] is
  such that the divisor class corresponding to
  $\left[\nu \circ (p_{1},p_{2})\right]^{*}\mc{O}(1)$ is $H +
  (d-3)h$. Furthermore, $\nu \circ (p_{1},p_{2})$ is birational onto
  its image, and its image is precisely $\P \Orb(C_{BN})$.

  Therefore, we conclude that the degree of $\P \Orb(C_{BN})$ is given
  by the intersection number $\int_{\Delta}(H+(d-3)h)^{8}$ on
  $\Delta$.

  Since $h^{3}=0$ on $\Delta$, this latter intersection number is
  equal to
\begin{align*}
  \int_{\Delta} H^{8} + 8(d-3)H^{7}h + {8 \choose 2}(d-3)^{2}H^{6}h^{2}.
\end{align*}

The numbers $a,b,c$ appearing in the statement of the lemma correspond
to the monomials $H^{8}, H^{7}h, H^{6}h^{2}$ evaluated on
$\Delta$. Each monomial is straightforwardly interpreted as the
solution to certain enumerative problems:
\begin{align*}
  H^{8}&=2\# \{\text{singular cubics through 8 points}\},
  \end{align*}
  where the coefficient of 2 arises because
  $p_{1}: \Delta \to \mb{P}\on{Sym}^{3}V^{\vee}$ is $2:1$ onto its
  image. Furthermore,
  \begin{align*}
    H^{7}h&=\# \{\text{nodal cubics through 7 points with a nodal branch line containing a fixed 8th point}\}\\
    H^{6}h^{2}&=\#\{\text{nodal cubics through 6 points with specified nodal branch line}\}.
\end{align*}
This proof of the lemma is complete, after noting that the exact value
of $a$ comes from the fact that there are twelve nodal cubics in a
general pencil of cubics \cite{W08}.
\end{proof}

Obviously, our task is now to establish the numbers $b,c$ in
\Cref{lem:predegreepolyCBN}.

\begin{lem}
\label{abcsum}
The sum $a+b+c$ in \Cref{lem:predegreepolyCBN} is $308$.
\end{lem}

\begin{proof}
  To compute $a+b+c$, we must specialize the general formula in
  \Cref{lem:predegreepolyCBN} to the case $d=4$. In this case,
  $C_{BN}$ is also the curve $C_{D_6}$. To calculate the degree of
  $\P \Orb(C_{D_6})$ in the space $\mb{P}\on{Sym}^{4}V^{\vee}$ of
  quartic plane curves, we apply \Cref{A6E6orbit} together with
  \Cref{projectivethom}. Explicitly, we take
\begin{align*}
  \frac{1}{\#\on{Aut}(C_{D_6})}\p_{C_{D_6}}&= 64 (18 c_1^6 + 33 c_1^4 c_2 + 12 c_1^2 c_2^2 - 85 c_1^3 c_3 - 11 c_1 c_2 c_3 - 
                                             7 c_3^2)
\end{align*}
from \Cref{A6E6orbit}, then make the substitution $c_1\mapsto u+v+w$,
$c_2\mapsto uv+uw+vw$, $c_3\mapsto uvw$ followed by the substitutions
$u\mapsto u-\frac{H}{4}$, $v\mapsto v-\frac{H}{4}$,
$w\mapsto w-\frac{H}{4}$, and then extract the coefficient of
$H^6$. The end result is $308$.
\end{proof}

\begin{prop}
\label{5cubics}
The coefficient $c$ in \Cref{lem:predegreepolyCBN} is
$5 \cdot {8 \choose 2}$.
\end{prop}

\begin{proof}
  By \Cref{lem:predegreepolyCBN} $c$ amounts to the following
  enumerative fact: Fix $6$ general points $p_{1}, ... p_{6}$ in
  $\P V$ and fix a general line $L \subset \P V$.  Then there are $5$
  cubics containing the points $p_{i}$ which are singular at a point
  on $L$ and meeting $L$ with multiplicity $\geq 3$ at the singular
  point.

  To prove this fact, we re-express the statement as the degree of the
  degeneracy locus of a map between two rank $4$ vector bundles
  \[e: \mathcal{A} \to \mathcal{B}\] on the line $L$. We omit the
  transversality argument implicit in this re-expression, and leave it
  to the reader.

  The vector bundle $\mathcal{A}$ is simply the trivial vector bundle
  whose fiber at any point $p \in L$ is the vector space
  \[H^{0}\left(\P V, \mc{I}_{p_1, \ldots, p_{6}}(3)\right)\] of cubic
  curves containing the $6$ points $p_{1}, ..., p_{6}$.

  We next describe the second vector bundle $B$ as a certain bundle of
  jets.  For each point $p \in L$, let
  \[\mc{J}_{p} \subset \mc{O}_{\mb{P}V}\] denote the ideal defining
  the divisor $3p$ in $L$, and let
  \[\mathfrak{m}^{2}_{p} \subset \mathcal{O}_{\mb{P}V}\] denote the
  square of the ideal sheaf of $p$. Let \[W_{p} \subset Z_{p}\] denote
  the subschemes defined by $\mc{J}_{p}$ and
  $\mc{J}_{p} \cap \mathfrak{m}^{2}_{p}$ respectively.

  We first define $\mathcal{B}'$ to be the rank $3$ vector bundle on
  $L$ whose fiber at a point $p$ is given by:
  \begin{align*}
  \mathcal{B}'|_{p} = H^{0}(\P V, \mc{O}(3))/H^{0}(\P V, \mathcal{J}_{p}(3)),
\end{align*}

Finally, we define $\mathcal{B}$ to be the rank $4$ vector bundle on
$L$ whose fiber at a given point $p$ is
\begin{align*}
  \mc{B}|_{p} = H^{0}(\P V, \mc{O}(3))/H^{0}(\P V, (\mathcal{J}_{p} \cap \mathfrak{m}_{p}^{2})(3)).
\end{align*}

Our next task is to compute the degree of $\mathcal{B}$.  The quotient
$\mc{J}_{p}/\left(\mc{J}_{p} \cap \mathfrak{m}^{2}_{p}\right) $ can
naturally be identified with the restriction of the conormal space
$\left(\mc{I}_{L}/\mc{I}_{L}^{2}\right)|_{p}$: In local affine
coordinates $(x,y)$, if $L$ is the line $x=0$ and $p$ is the origin,
then
$\mc{J}_{p} = (x,y^{3}), \mc{J}_{p} \cap \mathfrak{m}^{2}_{p} =
(x^{2},xy,y^{3})$, and
$\mc{J}_{p}/\mc{J}_{p} \cap \mathfrak{m}^{2}_{p}$ is generated by
$\bar{x}$, the local generator for
$\left(\mc{I}_{L}/\mc{I}_{L}^{2}\right)|_{p}$.

Putting these observations together, we obtain a short exact sequence
of vector bundles:
\begin{align}
  0 \to \mc{I}_{L}/\mc{I}^{2}_{L} \otimes \mc{O}_{L}(3) \to \mathcal{B} \to \mathcal{B'} \to 0.
\end{align}

Therefore, the degree of $\mathcal{B}$ (as vector bundle on $L$) is
equal to the degree of $\mathcal{B}'$ plus the degree of the line
bundle $\mc{I}_{L}/\mc{I}^{2}_{L} \otimes \mc{O}_{L}(3)$.  The latter
clearly has degree $2$.  $\mathcal{B}'$ is the standard second order
jet bundle for the line bundle $\mc{O}_{L}(3)$, and is easily seen to
have degree $3$. Therefore, the degree of $\mathcal{B}$ is $5$.

Furthermore, in reference to the affine coordinates above, notice that
a general cubic polynomial in the ideal $(x^{2},xy,y^{3})$ is
precisely a cubic singular at $(0,0)$ having one branch being the line
$x=0$.

At last, we take the map $e: A \to B$ induced by the natural quotient
map over each point $p$.  By the observation in the previous
paragraph, the locus of points where $e$ drops rank is precisely the
number $c$.  Since $A$ is trivial, the number of points where $e$
drops rank is the degree of $B$, which is $5$.  The lemma now follows.
\end{proof}

\begin{proof}[Proof of \Cref{prop:predegreepolyCBN}]
  Since $a=24$ and $c=5\cdot 28=140$ (from \Cref{5cubics}), it follows
  that $b=144$ thanks to \Cref{abcsum}, and the proposition is proved.
\end{proof}

\subsection{Degeneration to a cusp}

Recall that a cusp singularity on a plane curve is called
\emph{ordinary} if no line meets it with multiplicity $\geq 4$. Let
$U \subset \mathbb{A}^{1}$ be a suitable neighborhood of $0$ with
coordinate $t$.  As in \Cref{nodedegeneration}, we will write $R$ for
the coordinate ring of $U$, with valuation $v$ corresponding to $0$,
and we abuse notation and use fractional powers of $t$ to indicate
base changes $s^{n}=t$.

Throughout this subsection, we assume
\[\alpha_{\on{cusp}}: U \to \P \on{Sym}^{d}V^{\vee}\] is a family of curves
defined by a degree $d$ homogeneous form $F(X,Y,Z)$ with coefficients
in $R$, obeying:
\begin{enumerate}
\item The curve $C_{0} := \alpha_{\on{cusp}}(0)$ has an ordinary cusp
  singularity at $[0:0:1]$ meeting the line $Y=0$ to order $3$,
\item the total space $\mathcal{C}$ of the family is smooth at the
  cusp point on $C_{0}$.
\end{enumerate}

\begin{lem}
\label{cuspmatrix}
Maintain the setting above, and let $\gamma(t)$ denote the family of
matrices
\[\begin{pmatrix}
    t^{-1/3} & 0 & 0 \\
    0 & t^{-1/2} & 0 \\
    0 & 0 & 1\\
  \end{pmatrix}
  \]
  Then the curve $\alpha_{\on{cusp}}^{\gamma}(0)$ is defined by the
  equation
\begin{align*}
    \lim_{t\to 0}t^{-1}(F(t^{\frac{1}{3}}X,t^{\frac{1}{2}}Y,Z))
\end{align*}
and is projectively equivalent to the curve
$C_{\on{flex}}[0]: Z^{d-3}(X^3+Y^2Z+Z^3) = 0$. (See
\Cref{specialcurves}.)
\end{lem}

\begin{proof}
  The proof runs exactly parallel to the proof of \Cref{nodematrix},
  so we omit it.
\end{proof}

\begin{rmk}
  \label{rmk:cuspkernel}
  Notice that the limit endomorphism $\gamma(0)$ has as kernel space
  the line $Y=0$, i.e. the cuspidal tangent of the cusp of $C_0$.
  Meanwhile, the $t=0$ limit of $\gamma^{-1}$ has as its image the
  cusp point $[0:0:1]$.
\end{rmk}

\subsection{Degenerating the orbit of $C_{\on{flex}}[j]$}
\label{sec:jdegeneration}

In this subsection, we will study how $\P \Orb(C_{\on{flex}}[j])$
specializes as we vary the $j$-invariant to $\infty$.  The outcome
will be the proof of \Cref{prop:CANCBN}.  In order to proceed, we
introduce and study particular variety which will induce the
specialization implicit in \Cref{prop:CANCBN}.

\begin{defn}
  \label{def:W}
  Define
  \[\mathcal{W} \subset \P \on{Sym}^{3}V^{\vee} \times \P V^{\vee}
    \times \P V\] to be the closed subvariety parametrizing triples
  $(C,L,p)$ where $L$ is a line meeting the cubic curve $C$ at the
  point $p$ with multiplicity $\geq 3$.  Let $p_{1}, p_{2}, p_{3}$
  denote the projections of $\mathcal{W}$ to the respective factors.
  Finally, we let $H_{\on{curve}}, H_{\on{line}}$ and $H_{\on{point}}$
  denote the pullbacks of hyperplane classes under the maps
  $p_{1}, p_{2}, p_{3}$ respectively.
\end{defn}

It is easy to see that $\mathcal{W}$ is smooth: Indeed, $\mathcal{W}$
is a projective bundle over the smooth $3$-dimensional incidence
variety parametrizing pairs $(L,p)$ with $p \in L$.

\begin{defn}
  \label{def:j}
  We let \[J: \mathcal{W} \dashrightarrow \P^1\] denote the rational
  map which sends a triple $(C,L,p)$ to the $j$-invariant $j(C)$.
\end{defn}

The $j$-invariant is well-defined (and finite) for smooth cubics.  It
takes on the value $\infty$ for a cubic with a single node, for a
cubic with two nodes (i.e. a smooth conic union a line), and for a
cubic with three nodes (i.e. a union of three non-concurrent lines).
For any other cubic curve, the $j$-invariant is not defined.

For each $j \in \P^{1}$ let
\begin{align}
  \label{WjW}
  \mathcal{W}_{j} \subset \mathcal{W}
  \end{align}
  denote the divisor (possibly non-reduced) given by $J^{-1}(j)$.
  This is well-defined, as the locus of indeterminacy of $J$ has
  codimension at least $2$.  $J$ determines a rational equivalence of
  cycles:
  \[[\mathcal{W}_{j}] \sim [\mathcal{W}_{\infty}].\]

  Our short-term objective is to understand the irreducible components
  of $\mathcal{W}_{\infty}$.  We leave it to the reader to check that
  $\mathcal{W}_{\infty}$ has precisely three irreducible components:

\begin{itemize}
\item [$\mathcal{W}_{BN, \infty}$]: This is the closure of the locus
  of triples $(C,L,p)$ where $C$ has a unique node at $p$ and $L$ is
  one of the branches at the node,
\item [$\mathcal{W}_{AN, \infty}$]: This is the closure of the locus
  of those triples $(C,L,p)$ such that $C$ has a unique node, $p$ is a
  (smooth) flex point of $C$, and $L$ is the flex line at $p$.
\item [$\mathcal{W}_{\on{conic}}$]: This is the closure of the locus
  of those triples $(C,L,p)$ such that $C = Q \cup L$, where $Q$ is a
  smooth conic, and $p \in L$ is a general point.
\end{itemize}

Thus, for any finite $j \in \P^{1}$, we $J$ provides a
$GL(V)$-equivariant rational equivalence:
\begin{align}
  \label{eq:Jrat}
  [\mathcal{W}_{j}] \sim t_{BN} \cdot [\mathcal{W}_{BN, \infty}] + t_{AN} \cdot [\mathcal{W}_{AN, \infty}] + t_{\on{conic}} \cdot [\mathcal{W}_{\on{conic}}]. 
\end{align}

We must determine the two multiplicities $t_{BN},t_{AN}$ in
particular, because $[\mathcal{W}_{\on{conic}}]$ will not matter in
the final analysis.

\begin{prop}
  \label{prop:tBNtAN}
  In \eqref{eq:Jrat}, $t_{BN} = 3$ and $t_{AN} = 1$.
\end{prop}

\begin{proof}
  Recall from \Cref{def:W} the class $H_{\on{curve}}$ on $\mathcal{W}$
  corresponding to the line bundle $p_{1}^{*}(\mathcal{O}(1))$.  Upon
  intersecting both sides of \eqref{eq:Jrat} by $H_{\on{curve}}^{8}$
  we get:
  \begin{align}
    \label{eq:intW}
    \int_{\mathcal{W}} H_{\on{curve}}^{8} \cdot  [\mathcal{W}_{j}]  = t_{BN} \int_{\mathcal{W}}H_{\on{curve}}^{8} \cdot   [\mathcal{W}_{BN, \infty}] + t_{AN} \int_{\mathcal{W}}H_{\on{curve}}^{8} \cdot  [\mathcal{W}_{AN, \infty}] + 0.  
  \end{align}
  The $0$ arises because the cycle $\mathcal{W}_{\on{conic}}$ gets
  contracted under the projection $p_{1}$.  Now, we simply determine
  the three integrals above.
  \begin{enumerate}
  \item
    $\int_{\mathcal{W}} H_{\on{curve}}^{8} \cdot [\mathcal{W}_{j}] =
    12 \times 9$ because, in a general pencil of cubic curves, a given
    $j$-value arises $12$ times, and each smooth cubic has $9$ flexes.

  \item
    $\int_{\mathcal{W}}H_{\on{curve}}^{8} \cdot [\mathcal{W}_{BN,
      \infty}] = 12 \times 2$ because, in a general pencil of cubics,
    there are $12$ singular cubics and each one has two branch lines,

  \item
    $\int_{\mathcal{W}}H_{\on{curve}}^{8} \cdot [\mathcal{W}_{AN,
      \infty}] = 12 \times 3$ because, a general pencil of cubics has
    $12$ singular members, and each has $3$ (smooth) flex points.
  \end{enumerate}

  Therefore, we conclude that \[108 = 24t_{BN} + 36t_{AN}.\] The only
  positive integer solution to this latter equation is
  $t_{BN} = 3, t_{AN} = 1$, and the proposition is proved.
\end{proof}

\begin{proof}[Proof of \Cref{prop:CANCBN}]\label{proof:CANCBN}
  The rational equivalence induced by $J$ in \eqref{eq:Jrat} is
  evidently $GL(V)$-equivariant. We push forward this rational
  equivalence to $\P \on{Sym}^{d}V^{\vee}$ via the map
  $\varphi: \mathcal{W} \to \on{Sym}^{d}V^{\vee}$ sending $(C,L,p)$ to
  the curve $C \cup (d-3)L$, and after using \Cref{prop:tBNtAN}, the
  proposition follows.
\end{proof}

\subsection{The degree of $\P \Orb(C_{\on{flex}}[j])$} Next, we
compute the degree of the orbit closure
$\P \Orb(C_{\on{flex}}[j]) \subset \P \on{Sym}^{d}V^{\vee}$, where $j$
is {\sl general}.  Again, although this can be computed in principle
using the algorithm of Aluffi and Faber \cite{AF00}, we have decided
to proceed independently, providing a further check.

First, we record the analogue to \Cref{lem:predegreepolyCBN}.
\begin{lem}
  \label{lem:Cflexpoly} As a function of $d$, the degree of the orbit
  closure $\P \Orb(C_{\on{flex}}[j])$, $j$ general, is a quadratic
  polynomial $a + b\cdot(d-3) + c \cdot (d-3)^{2}$, where the
  coefficients $a,b,c$ are the answers to the following enumerative
  problems:
\begin{align*}
  a &= 9\cdot 12\# \{\textrm{Cubics through 9 points} \}= 108\\
  b &= 12\cdot \binom{8}{1}\# \{\text{Cubics through 8 points with flex line containing a fixed 9th point}\}\\
  c &= 12\cdot \binom{8}{2} \#\{\text{Cubics through 7 points flexed at a specified line}\}
\end{align*}
\end{lem}

\begin{proof}\label{proof:Cflex}
  Recall the schemes $\mathcal{W}$ and $\mathcal{W}_j$ from
  \eqref{WjW}, and the classes $H_{\on{curve}}$ and $H_{\on{line}}$
  from \Cref{def:W}. The variety $\mathcal{W}$ affords the natural map
  \[\varphi: \mathcal{W} \to \P \on{Sym}^{d}V^{\vee}\]
  defined by sending $(C,L,p)$ to the curve $C \cup (d-3)L$, and the
  pullback of a hyperplane under this map has class
  $H_{\on{curve}} + (d-3)H_{\on{line}}$.

  Our objective is to calculate the degree of the image
  $\varphi(\mathcal{W}_{j})$, as this is precisely the orbit closure
  of $C_{\on{flex}}[j]$.  Thus, it suffices to compute
  $\int_{\mathcal{W}}(H_{\on{curve}}+(d-3)H_{\on{line}})^{8} \cdot
  [W_j]$ in the Chow ring of $\mathcal{W}$.

  Since
  $H_{\on{line}}^{3}=0$, we get that the degree of the orbit closure
  of $C_{\on{flex}}[j]$ is:
  \begin{align}
    \label{eq:polyflex}
    \int_{\mathcal{W}} \left(H_{\on{curve}}^{8}\cdot [\mathcal{W}_j] + 8(d-3)H_{\on{curve}}^{7}H_{\on{line}}\cdot [\mathcal{W}_j] + {8 \choose 2}(d-3)^{2}H_{\on{curve}}^{6}H_{\on{line}}^{2}\cdot [\mathcal{W}_j] \right).
  \end{align}

  Next, we observe that the divisor class of $\mathcal{W}_j$ is
  $12 H_{\on{curve}}$, because the degree of the divisorial locus in
  $\P \on{Sym}^{3}V^{\vee}$ consisting of the closure of plane cubics
  with prescribed generic $j$-invariant is $12$.  Therefore, the
  degree of the orbit closure of $C_{\on{flex}}[j]$ is
  \begin{align*}
    12 \int_{\mathcal{W}}\left( H_{\on{curve}}^{9} + 8(d-3)H_{\on{curve}}^{8}H_{\on{line}} + {8 \choose 2}(d-3)^{2}H_{\on{curve}}^{7}H_{\on{line}}^{2} \right).
  \end{align*}

  The lemma now follows by interpreting the three intersection numbers
  \[\int_{\mathcal{W}}H_{\on{curve}}^{9}, \int_{\mathcal{W}}
    H_{\on{curve}}^{8}H_{\on{line}}, \int_{\mathcal{W}}
    H_{\on{curve}}^{7}H_{\on{line}}^{2}\] as the quantities appearing
  in the descriptions of $a, b$ and $c$ in the statement of the lemma.
  \end{proof}

  Our next task is to determine the coefficients $a,b,c$ in
  \Cref{lem:Cflexpoly}.

  \begin{lem}
    \label{lem:bCflex}
    There are $9$ cubic curves passing through eight general points
    and having a flex line containing a general fixed ninth point,
    i.e. \[\int_{\mathcal{W}}H_{\on{curve}}^{8}H_{\on{line}} = 9.\]
  \end{lem}

  \begin{proof}
    Let $\Lambda \subset \P \on{Sym}^{3}V^{\vee}$ denote the Hesse
    pencil
    \[s(X^{3} + Y^3 + Z^3) + tXYZ = 0, \, [s:t] \in \P^1.\] Recall
    that the $9$ base points of the Hesse pencil consist of the $9$
    flexes of every smooth cubic in $\Lambda$.  At each base point $p$
    of the pencil, the flex lines of the cubic curves in the pencil in
    sweep out a pencil of lines in $\P V^{\vee}$.  Therefore, a
    general point $x$ in $\P^{2}$ is contained in exactly $9$ flex
    lines of members of the Hesse pencil, one per basepoint.

    Thus, if we use the pullback of the Hesse pencil $\Lambda$ to
    represent the curve class $H_{\on{curve}}^{8}$ on $\mathcal{W}$,
    we deduce the lemma.
  \end{proof}

  \begin{lem}
    \label{lem:cCflex} There are $3$ cubic curves passing through $7$
    general points and possessing a particular line as flex line,
    i.e. \[\int_{\mathcal{W}}H_{\on{curve}}^{7}H_{\on{line}}^{2}=3.\]
  \end{lem}

  \begin{proof}
    Let $L \subset \P V$ be a fixed line, and suppose
    $p_{1}, ..., p_{7} \in \P V$ are general points.  Then the net of
    cubic curves containing the points $p_{i}$ restricts to a general
    net in the linear system $|\mc{O}_{L}(3)|$.  A general such net
    maps $L$ to a nodal cubic in $\P^{2}$. This nodal cubic has
    exactly $3$ (smooth) flex points.  The points on $L$ corresponding
    to these flex points, in turn, provide the solutions to the
    enumerative problem in the statement of the lemma.
  \end{proof}

  \begin{cor}
    \label{cor:predegreeCflex}
    The degree of $\P \Orb(C_{\on{flex}}[j])$, $j$-general, is
    $12 \left(9 + 72(d-3) + 84(d-3)^2\right)$. The predegree of
    $C_{\on{flex}}[j]$ is $24 \left(9 + 72(d-3) + 84(d-3)^2\right)$.
  \end{cor}

  \begin{proof}
    Combine \Cref{lem:Cflexpoly}, \Cref{lem:bCflex},
    \Cref{lem:cCflex}. The second statement on the predegree follows
    from the fact that the curve $C_{\on{flex}}[j]$ has an order $2$
    automorphism group, since the generic elliptic curve has an order
    $2$ automorphism group.
  \end{proof}
  
  \subsection{Proof of \Cref{nodecusp}}
  
  We now have all ingredients for the proof of \Cref{nodecusp}.
  Recall the setup in \Cref{sec:frameworknodecusp}.

\begin{proof}[Proof of \Cref{nodecusp}]

Let $p_{i} \in C_0$, $i=1, \dots, \delta$ and $q_{j} \in C_0$,
$j=1, \dots, \kappa$ denote the ordinary nodes and cusps,
respectively.  For each $i = 1, \dots, \delta$ or
$j = 1, \dots, \kappa$ separately, upon conjugating by a suitable
invertible change of coordinates in $PGL(V)$, we obtain $1$-parameter
families $\gamma_{i,1}(t), \gamma_{i,2}(t)$ ,and $\gamma_{j}(t)$
appearing \Cref{nodematrix} and \Cref{cuspmatrix} accordingly.

Recall the variety
\[\mathcal{Y} \subset U \times \mathsf{Inv}(V) \times \P
  \on{Sym}^{d}V^{\vee}\] attached to the family $\alpha$. (See
\Cref{def:Inv} for the definition of $\mathsf{Inv}(V)$.)  Each
$\gamma_{i,1}(t), \gamma_{i,2}(t)$ and $\gamma_{j}(t)$ is full for
$\alpha$ due to \Cref{nodematrix} and \Cref{cuspmatrix} -- the curves
$C_{BN}$ and $C_{\on{flex}}[0]$ have finitely many
automorphisms. Thus, by \Cref{def:Agamma}, we get corresponding
irreducible components
$\mathcal{A}_{i,1}, \mathcal{A}_{i,2}, \mathcal{A}_{j}$ in the fiber
$\mathcal{Y}_0$.

By the $\mathsf{Inv}(V)$-adaptation of \Cref{gamma12}, and by the
statements about kernels and images found in \Cref{rmk:limitendoNode}
and \Cref{rmk:cuspkernel}, it follows that the irreducible components
$\mathcal{A}_{i,1}, \mathcal{A}_{i,2}, \mathcal{A}_{j}$ are all
pairwise distinct.  Therefore, the hypotheses of \Cref{principle2} are
met, and the statement follows.  (The coefficient of $2$ in
$-2 \delta \cdot p_{C_{BN}}$ arises because there are two components
$\mathcal{A}_{i,1}, \mathcal{A}_{i,2}$ corresponding to each node.)

It remains to establish compatibility of predegrees. For this, we once
again reference Aluffi and Faber, specifically \cite[Example 4.1 and
Example 5.2]{AF00} where the defect on the predegree caused by an
ordinary node and cusp are shown to be, respectively
  \begin{align*}
     24 (35 d^2 - 174 d + 213)&=2\cdot 3(24+144(d-3)+140(d-3)^2) \\ 
     72 (28 d^2 - 144 d + 183)&=24(9+72(d-3)+84(d-3)^2) 
  \end{align*}
  These two numbers, fortuitously, happen to be twice the predegree of
  $C_{BN}$ and the predegree of $C_{\on{flex}}[j]$ respectively by
  \Cref{prop:predegreepolyCBN} and \Cref{cor:predegreeCflex}.  The
  theorem follows from \Cref{principle2}.
\end{proof}


  \section{Computation of $\p_{C_{AN}}$ and $\p_{C_{BN}}$}
\label{CAN}
In light of the main results of \Cref{sec:nodecusp}, it is natural to
seek explicit formulas for $\p_{C_{AN}}$ and $\p_{C_{BN}}$. In this
section we provide a method for this computation, and apply it to the
specific case $d=4$. Recall that when $d=4$, $C_{BN}$ is equal to
$C_{D_{6}}$ -- our computation yields the same result as the
computation of $\p_{C_{D_{6}}}$ in \Cref{A6E6orbit}, thankfully.

\begin{prop}\label{prop:CAN}
  When $d=4$,
\begin{align*}
  \p_{C_{BN}}&=3 \cdot 64 (18 c_1^6 + 33 c_1^4 c2 + 12 c_1^2 c_2^2 - 85 c_1^3 c_3 - 11 c_1 c_2 c_3 - 
              7 c_3^2)\\
  \p_{C_{AN}}&=2 \cdot 192 (18 c_1^6 + 33 c_1^4 c_2 + 12 c_1^2 c_2^2 + 19 c_1^3 c_3 - 7 c_1 c_2 c_3 - 
              35 c_3^2).
\end{align*}
\end{prop}

\begin{proof}

  In the interest of brevity, we will indicate {\sl how} to arrive at
  the result, by breaking the calculation into several steps and
  explaining how each step is executed.

  First, the factor of $3$ in $\p_{C_{BN}}$ and $2$ in $\p_{C_{AN}}$
  come from automorphisms.  What we will actually describe, and what
  is equivalent by \Cref{projectivethom}, is the calculation of the
  classes $[\Orb(C_{BN})]_{GL(V)}$ and $[\Orb(C_{AN})]_{GL(V)}$,
  essentially using the same overall strategy of ``presentation'' used
  in the proof of \Cref{fourlines}.

  We begin with the variety $\mathcal{W}$ defined in \Cref{def:W},
  with its attendant divisor classes
  $H_{\on{curve}}, H_{\on{line}}, H_{\on{point}}$. We can regard
  $\mathcal{W}$ as an iterated projective bundle, by first forgetting
  $C$, then forgetting $p$, and then forgetting $L$ (to map to a
  point). Each of these projective bundles are given by the
  projectivization of a vector bundle over their associated base
  spaces, and so the Chow ring of $W$ is determined by the chern
  classes of these vector bundles in a standard way.  Furthermore, the
  Picard group of $\mathcal{W}$ is freely generated by the three
  divisor classes $H_{\bullet}$.

  Let
  \[\mathsf{Inc} \subset \P V^{\vee} \times \P V\] denote the standard
  {\sl incidence variety} parametrizing pairs $(L,p)$ where $p \in L$.
  Observe that $\mathcal{W}$ is a projective bundle over
  $\mathsf{Inc}$.

  We let

  \[\phi: \mathcal{W} \to \P \on{Sym}^{4}V^{\vee}\]
  denote the map sending a triple $(C,L,p)$ to the quartic curve
  $C \cup L$, and let $H$ denote the pullback of the hyperplane class
  via $\phi$.  Observe that $H = H_{\on{curve}} + H_{\on{line}}$.

  Finally, recall the irreducible divisors (specified a few lines
  after \eqref{WjW})
  $\mathcal{W}_{\infty,BN}, \mathcal{W}_{\infty,AN}$,
  $\mathcal{W}_{\on{conic}}$ in $\mathcal{W}$. We now provide the
  stages of the calculation.

  \underline{{\bf Steps of the calculation}}:
  \begin{itemize}
  \item [{\bf Step 1}:] The divisor class $[\mathcal{W}_{\on{conic}}]$
    is straightforward to compute because $\mathcal{W}_{\on{conic}}$
    has a simple description. Indeed, $\mathcal{W}_{\on{conic}}$ is a
    projective sub-bundle of the projective bundle
    $\mathcal{W} \to \mathsf{Inc}$.  It is straightforward to access
    this sub-bundle, and the conclusion is:
    \begin{align}
      \label{eq:conic}
      [\mathcal{W}_{\on{conic}}] = H_{\on{curve}} - 3H_{\on{point}} + 3H_{\on{line}}.
    \end{align}
  \item [{\bf Step 2}:] 
   Let
  \[\mathcal{R} \subset \mathcal{W}\] denote the ramification divisor
  of $p_{1}$ -- observe that
  \[p_{1}: \mathcal{W} \to \P \on{Sym}^{3}V^{\vee}\] is unramified
  over the locus parametrizing smooth cubic curves, and is also
  unramified along $\mathcal{W}_{AN,\infty}$. Applying the
  Riemann-Hurwitz formula (i.e. subtracting canonical classes) to
  $p_{1}$ yields the class of the ramification divisor:
  \begin{align}
    \label{eq:R}
    [\mathcal{R}] = 3H_{\on{curve}} + H_{\on{line}} + H_{\on{point}}.
  \end{align}
  $\mathcal{R}$ is supported on $\mathcal{W}_{BN, \infty}$ and
  $\mathcal{W}_{\on{conic}}$. Using the classical fact that a branch
  line of a node is the limit of three flexes,
  $\mathcal{W}_{BN, \infty}$ appears with multiplicity 2 in the
  ramification divisor $\mathcal{R}$. In light of this, by
  concentrating on the $H_{\on{curve}}$ coefficients in
  $[\mathcal{R}]$ and $[\mathcal{W}_{\on{conic}}]$, there is only one
  option for the class $[\mathcal{W}_{BN,\infty}]$:
    \begin{align}
      \label{eq:WBN}
      [\mathcal{W}_{BN,\infty}] = H_{\on{curve}} - H_{\on{line}} + 2H_{\on{point}}.
    \end{align}
  \item [{\bf Step 3}:] Next, to access $\mathcal{W}_{AN,\infty}$ we
    pull back the discriminant hypersurface
    \[\mathsf{Disc} \subset \P \on{Sym}^{3}V^{\vee}\] (parametrizing singular cubics) via
    $p_{1}$. The divisor $p_{1}^{*}(\mathsf{Disc})$ is clearly
    supported on $\mathcal{W}_{BN, \infty}, \mathcal{W}_{AN, \infty}$,
    and $\mathcal{W}_{\on{conic}}$. Since $p_{1}$ is locally $3:1$
    near a general point of $\mathcal{W}_{BN, \infty}$ and unramified
    at a general point of $\mathcal{W}_{AN,\infty}$ we conclude that
    $p_{1}^{*}(\mathsf{Disc}) = 3 [\mathcal{W}_{BN, \infty}] +
    [\mathcal{W}_{AN,\infty}] + z \cdot [\mathcal{W}_{\on{conic}}]$
    for some positive integer $z$.

  \item [{\bf Step 4}:] The integer $z$ above is $2$. This follows
    from the local calculation that a general pointed curve (spectrum
    of a DVR) $(T,0) \subset \mathcal{W}$ intersecting
    $\mathcal{W}_{\on{conic}}$ transversely (at $0 \in T$) at a
    general point $(Q \cup L, L, p)$ will produce a curve in
    $\P \on{Sym}^{3}V^{\vee}$ meeting $\mathsf{Disc}$ with
    multiplicity $2$ at $0$. Here $Q$ and $L$ are a general conic and
    line, respectively. The multiplicity $2$ comes from the fact that
    the total space of the family of curves parametrized by $T$ is
    smooth at the two nodes $Q \cap L$ of the curve $Q \cup L$.

  \item [{\bf Step 5}:] Combining the previous steps, we conclude:
    \begin{align}\label{eq:WAN}
      [\mathcal{W}_{AN, \infty}] = 7H_{\on{curve}} - 3H_{\on{line}}
    \end{align}

  \item [{\bf Step 6}:] We now observe that all objects and
    constructions in the previous step were $GL(V)$-equivariant.  This
    means it is possible to do exactly the same calculations in the
    {\sl relative setting}: $\mathcal{V} \to B$ is now an arbitrary
    rank $3$ vector bundle over a variety $B$.  $\mathcal{W}$ is
    replaced by
    \[ \mathcal{W}_{\mathcal{V}} \subset \P
      \on{Sym}^{3}\mathcal{V}^{\vee} \times \P \mathcal{V}^{\vee}
      \times \P \mathcal{V},\] the divisor classes $H_{\bullet}$
    relativize as the pullbacks of the $\mathcal{O}(1)$ classes on
    each of the bundles
    $\on{Sym}^{3}\mathcal{V}^{\vee}, \P \mathcal{V}^{\vee}, \P
    \mathcal{V}^{\vee}$, etc...

    Let $c_{1}, c_{2}, c_{3}$ denote the chern classes of
    $\mathcal{V}$. By performing the calculations now in the relative
    setting, the analogous divisor classes are:
    \begin{align}
      \label{eq:relativeW}
      H = H_{\on{curve}} + H_{\on{line}},\\ 
      [\mathcal{W}_{BN,\infty}]_{\mathcal{V}} = H_{\on{curve}} - H_{\on{line}} + 2H_{\on{point}}\\
      [\mathcal{W}_{AN, \infty}]_{\mathcal{V}} = 7H_{\on{curve}} - 3H_{\on{line}} - 6c_{1}
    \end{align}
    These are elements in the Chow ring
    $A^{\bullet}(\mathcal{W}_{\mathcal{V}})$.

  \item[{\bf Step 7}:] Recall that our objective is to compute the
    classes $[ \Orb({C_{BN}})]_{GL(V)}$ and
    $[\Orb({C_{AN}})]_{GL(V)}$.  To get there from the previous
    step, we use the exact same trick as in \Cref{fourlines} and
    \cite[Theorem 3.1]{FNR06}.  Letting
    \[\phi: \mathcal{W}_{\mathcal{V}} \to \P
      \on{Sym}^{4}\mathcal{V}^{\vee}\] denote the relative version of
    the map sending $(C,L,p)$ to $C \cup L$, we must pull back the class
    \[\alpha =
      H^{14}+c_1(\on{Sym}^4\mc{V}^{\vee})H^{13}+\cdots+c_{14}(\on{Sym}^4\mc{V}^{\vee})\]
    via $\phi$, intersect with
    $[\mathcal{W}_{BN,\infty}]_{\mathcal{V}}$ (or
    $[\mathcal{W}_{AN,\infty}]_{\mathcal{V}}$ respectively) and
    push-forward to $B$ (using the Leray relation for projective
    bundles).  This yields $[\Orb(C_{BN})]$ and $[\Orb({C_{AN}})]$,
    and concludes the calculation, after applying the change of
    variables \Cref{projectivethom}.

    \end{itemize}

\end{proof}

\begin{rmk}
  The only place $d=4$ was used in the the proof of \Cref{prop:CAN}
  was the definition of class $\alpha$ and the pullback map
  $A^{\bullet}(\mb{P}(\on{Sym}^4\mc{V}^\vee))\to
  A^{\bullet}(W_{\mc{V}})$, meaning that the above provides a recipe
  to get the formulas for $\p_{C_{AN}}$ and $\p_{C_{BN}}$ for all $d$,
  but we have not tried to use it to arrive at a closed expression.
\end{rmk}


\section{Degenerations III: Quartic Plane Curves}
In this section, we finally specialize all the way to the setting
$r=2, d=4$ of quartic curves. We thoroughly study the classical
degeneration to a double conic.  Then we investigate the effect of
acquiring a hyperflex on $\p_{C}$. (We believe the hyperflex
specialization can be analyzed in arbitrary degree, but the algorithm
in \cite{AF00} is too complicated for us to apply with sufficient
confidence.)

\subsection{Degeneration to the double conic}

In this section, we study how $\p_{C}$ changes as a general smooth
quartic $C$ specializes to a double conic.

\subsection{Preliminary lemmas}

\begin{lem}
\label{interpolate}
Let $Q \subset \P V$ be a smooth conic and $p\in Q$ a point. Let $p_1$
and $p_2$ (respectively $p_1$) be points of $\mb{P} V$ so that $p_1$,
$p_2$ and $p$ are not collinear (respectively not lying on the tangent
line to $Q$ at $p$).  Then, there exists a unique smooth conic
$Q' \subset \P V$ meeting $Q$ at $p$ with multiplicity 3 (respectively
4) and containing $p_1$ and $p_2$ (respectively $p_1$).
\end{lem}

\begin{proof}
  Let $Z \subset Q$ be the scheme of length 3 (respectively 4)
  supported at $p\in Q$.  By counting conditions, for both $n=3,4$ we
  see that there is a conic $Q'$ containing $Z$,
  $p_1,\ldots,p_{5-n}$. If $n=3$, then $Q'$ cannot be a double line
  (since $Z$, $p_1$, $p_2$ are not contained in a line), nor can it be
  the union of two distinct lines (since $Z$ is not contained in a
  line, and also $p$, $p_{1}$, $p_{2}$ are not collinear). Therefore
  the conic $Q'$ is smooth. It is unique because of Bezout's theorem.

  If $n=4$, then $Q'$ cannot be a double line since the underlying
  line would then have to be tangent to $Q$ at $p$, but the tangent
  line does not pass through $p_1$ by assumption. We also cannot have
  $Q'$ be the union of two distinct lines or else $Q'$ can only meet
  $Q$ at $p$ with maximum multiplicity 3. Therefore $Q'$ is smooth,
  and again Bezout's theorem provides uniqueness.
\end{proof}

\begin{lem}
\label{coniclemma}
Let $n$ be such that $3\leq n\leq 7$ and let $C \subset \P V$ be a
general quartic curve with an $A_n$ singular point. Then, there exists
a smooth conic meeting $C$ at $C$'s singular point with multiplicity
$n+1$ and meeting $C$ transversely at $7-n$ other points.
\end{lem}

\begin{proof}
  We will do this case by case. Let $p\in C$ be the singular point.

  \begin{itemize}
  \item [{\bf Case $n=3$}:] For the case of a tacnode $n=3$, the
    desired conic is required to pass through $p$ with a specified
    tangent direction and otherwise intersect $C$ transversely. There
    is a $3$-dimensional linear system
  \[\Lambda \subset \P \on{Sym}^{2}V^{\vee}\]
  of conics passing through $p$ with a specified tangent direction. In
  $\Lambda$, the set of conics that intersect $C$ at 4 distinct
  points, apart from $p$, is a {\sl nonempty} open set, because the
  union of the unique line passing through $p$ in the specified
  tangent direction together with a line intersecting $C$ transversely
  provides an example.

  Since the set of smooth conics in $\Lambda$ is also nonempty, there
  exists a smooth conic passing through $p$ in the specified tangent
  direction and intersecting $C$ at four other points transversely,
  finishing this case..

\item [{\bf Case $n=4$}:] For the case $n=4$, (a {\sl ramphoid cusp})
  we need to resort to actual equations. The space of conics meeting
  $C$ at $p$ to order 5 is the same as the space of conics containing
  a particular length 3 curvilinear scheme $Z$, and we can assume
  $p=[0:0:1]$ and $Z$ is given by the length 3 neighborhood of
  $X^2+YZ$ around $p$. Since the condition we are trying to realize is
  open, we can specialize $C$ while preserving $p$ and $Z$, and it
  suffices to prove the result for a specialization of $C$. To this
  end, consider the particular rational quartic curve $C_0$ given by
$$(X^2+YZ)^2+X^3Y=0,$$ which has a rhamphoid cusp ($A_4$ singularity) at $[0:0:1]$ and an ordinary cusp at $[0:1:0]$.

Consider the conic $Q_{a,b}$ given by \[X^2+YZ+aXY+bY^2=0.\] If we
reduce the equation of $C_0$ modulo the equation of $Q_{a,b}$, then we
get the same result as reducing the equation
$(X^3+a^2X^2Y+2abXY^2+b^2Y^3)Y$. Therefore, the intersection of $C_0$
with $Q$ (as a scheme) is also given by the intersection of the union
of 4 lines through $p=[0:0:1]$ with $Q_{a,b}$. One of these lines, the
line given by $Y=0$, is tangent to $Q_{a,b}$ at $p$, so it suffices to
check the remaining three lines are distinct for general $a,b$. This
can be shown by noting that the discriminant of the cubic polynomial
$X^3+a^2X^2Y+2abXY^2+b^2Y^3$ does not vanish identically (indeed it is
not even homogenous). Thus, for general $a,b$, $Q_{a,b}$ furnishes the
conic we want.

\item [{\bf Cases $n=5,6$}:]For the cases $n=5,6$, we use
  \Cref{interpolate}. In both cases, we have a curvilinear scheme $Z$
  of length $n-2$ contained in a conic, and we need to find a smooth
  conic containing $Z$ and passing through $7-n$ distinct other points
  of $C$. If $n=5$, then it suffices to pick the remaining 2 points
  $p_1,p_2$ of $C$ so that $p$, $p_1$, and $p_2$ do not all lie on a
  line. If $n=6$, it suffices to pick the remaining point $p_{1}$ to
  not be contained in the tangent line to $Z$.
\end{itemize}
\end{proof}

\subsection{How $C_{A_{6}}$ arises}

In this subsection, let $F(X,Y,Z) \in \on{Sym}^{4}V^{\vee}$ define a
smooth quartic containing the point $[0:0:1]$ and meeting the conic
$Q(X,Y,Z) = X^{2}+YZ$ transversely there.  Next, let
\[\alpha: U \to \P \on{Sym}^{4}V^{\vee}\] be the family defined by
\[t^{3}F + Q^{2}\] where $t$ is a coordinate on $U \subset \mb{A}^{1}$ a
neighborhood of $0$.

Finally, define the $1$-parameter family
\[\gamma: U^{\times} \to PGL(V)\] given by
\[\begin{pmatrix}
    t^{-1} & 0 & 0 \\
    0 & t^{-2} & 0 \\
    0 & 0 & 1\\
  \end{pmatrix}. \]

\begin{lem}
\label{doubleconic}
Let $F(X,Y,Z)$ cut out a quartic plane curve and let
$Q(X,Y,Z)=X^2+YZ$. Suppose $F$ and $Q$ meet transversely at
$[0:0:1]$. Then the curve $\alpha^{\gamma}(0)$ is defined by
\begin{align*}
  \lim_{t\to 0}t^{-4}(t^3F(tX,t^{2}Y,Z)+Q(tX,t^{2}Y,Z))
\end{align*}
and is projectively equivalent to $C_{A_6}$ (from
\Cref{def:specialcurves}.)
\end{lem}

\begin{proof}
Note $Q(tX,t^{2}Y,Z)^2=t^4Q(X,Y,Z)$. Also, the only coefficients of
$t^3F(tX,t^{2}Y,Z)$ whose vanishing order with respect to $t$ is at
most 4 are the coefficients of $Z^4$ and $Z^3 X$. Since $F$ vanishes
at $p=[0:0:1]$ by assumption, the coefficient of $Z^4$ is zero.

The tangent line to $\{Q=0\}$ at $p$ is given by $Y=0$. Since
$\{F=0\}$ is transverse to $\{Q=0\}$ at $p$, the coefficient of $Z^3X$
is nonzero. Therefore,
\begin{align*}
  \lim_{t\to 0}t^{-4}(t^3F(t^2X,tY,Z)+Q(t^2X,tY,Z)^2)=(X^2+YZ)^2+aZ^3X
\end{align*}
for some constant $a\neq 0 \in \k$, which is the unique, up to
projective equivalence, rational curve with an $A_6$ singularity
having a full dimensional orbit, as found in \cite[Example
5.4]{AF00}. This is the curve $C_{A_{6}}$
\end{proof}

\begin{rmk}
  \label{rmk:imageA6}
  Notice that the image of the $t \to 0$ limit endomorphism
  $\gamma^{-1}$ is the point $p$.
\end{rmk}

\begin{cor}
\label{cor24}
For a general quartic plane curve $C$,

\[ \P \Orb(C) \leadsto 8 \cdot 3 \cdot \P \Orb({C_{A_6}})\] and
therefore
\[\p_C = 8 \p_{C_{A_6}}.\]
\end{cor}

\begin{proof}
  Let $F(X,Y,Z)$ cut out $C$. Pick a conic intersecting $C$
  transversely in 8 points and let $Q(X,Y,Z)$ cut out the conic. Then,
  consider the family of curves over $\mb{A}^1$ given by
\begin{align*}
  t^3F(X,Y,Z)+Q(X,Y,Z)^2.
\end{align*}
Applying \Cref{doubleconic} (after conjugating by appropriate elements
in $PGL(V)$) yields eight $1$-parameter families
$\gamma_i: U\to PGL(V)$, $1\leq i\leq 8$, to be used in
\Cref{principle2}. The principle applies, thanks to the
$\mathsf{Inv}(V)$-variant of \Cref{gamma12} and the observation in
\Cref{rmk:imageA6}.

To conclude, we use either \cite[Example 5.4]{AF00} or
\Cref{A6E6orbit} to see that the predegree of the rational quartic
$C_{A_6}$ is $1785$, and $1785\cdot 8=14280$, which is the predegree
the orbit of a general quartic curve \cite{AF93}. Finally, the factor
of $3$ in front of $\P \Orb(C_{A_6})$ arises because
$\#\on{Aut}(C_{A_6})=3$.
\end{proof}

\begin{thm}
\label{smooth}
For any smooth quartic plane curve $C$ with no hyperflexes,
\begin{align*}
  \p_C = 8 \p_{C_{A_6}}.
\end{align*}
\end{thm}

\begin{proof}
  The result follows from combining \Cref{principle}, \Cref{cor24},
  and the fact that the predegree of a general plane quartic is the
  same as the predegree of any smooth quartic not possessing a
  hyperflex (\cite{AF93}).
\end{proof}

\begin{rmk}
  We remark that our usage of the predegree computations of Aluffi and
  Faber in \cite{AF93} can in principle be replaced by the explicit
  description of the semistable reduction of an $A_n$ singularity
  given in \cite{CL13}.
\end{rmk}

\begin{thm}
\label{Anorbit}
Let $C_{A_n}$ be a general curve with an $A_n$ singularity, where
$3\leq n\leq 6$. Then,
\begin{align*}
  \#\on{Aut}(C_{A_n}) \P \Orb({C_{A_n}})\leadsto (7-n) \cdot 3 \cdot \P \Orb({C_{A_6}})
\end{align*}
and therefore
\[\p_{C_{A_n}}=(7-n)\p_{C_{A_6}}.\]
\end{thm}


\begin{proof}
  By \Cref{coniclemma} we can find a smooth conic that meets $C_{A_n}$
  at its singular point to order $n+1$ and meets $C$ transversely at
  $7-n$ other points $p_1,\ldots,p_{7-n}$. Let $F(X,Y,Z)$ cut out
  $C_{A_n}$ and $Q(X,Y,Z)$ cut out the conic.

  Consider the family of quartic curves given by
\begin{align*}
    t^3 F(X,Y,Z)+Q(X,Y,Z)^2.
\end{align*}
Note, critically, that for {\sl general} $t$, we get a curve with an
$A_n$ singularity. From \Cref{doubleconic} we get, upon conjugating by
appropriate elements of $PGL(V)$, $7-n$ $1$-parameter families
$\gamma_i: U^{\times} \backslash\{0\}\to PGL_3$ to use in
\Cref{principle2}. The rest of the argument is just as in the proof of
\Cref{cor24}.

Applying \cite[Example 5.4]{AF00}, we find that the predegree of
$C_{A_n}$ is $(7-n)$ times the predegree of $C_{A_6}$ if $n\geq 3$,
and so \Cref{principle2} gives the result.
\end{proof}

\subsection{When a quartic acquires hyperflexes}
\label{sec:hyperflex}
Aluffi and Faber studied the situation of a smooth plane curve with no
hyperflexes specializing to a smooth curve possessing a hyperflex in
\cite[Theorem IV(2)]{AF93P}. We will analyze this degeneration, and
instead of using explicit equations as in \cite{AF93P} and as in our
previous degenerations, we found the need to proceed more
conceptually.  The ideas in this section will be familiar to those
readers having experience in the theory of limit linear series, though
we won't use any substantial part of that theory.

First, we need the following lemma characterizing $C_{E_{6}}$.

\begin{lem}
\label{E6}
Let $C \subset\mb{P}^2$ be a rational quartic possessing one $E_6$
singularity and two distinct flex points. Then, $C$ has a full
dimensional orbit, and is therefore projectively equivalent to
$C_{E_6}$.
\end{lem}

\begin{proof}
  We will show $\on{Aut}(C)$ is finite by showing that only a finite
  subgroup of $PGL_3$ preserves the flexes and the unique tangent
  branch at the $E_6$ singularity. Let $G$ be the component of
  $\on{Aut}(C)$ containing the identity.

  Without loss of generality, we can assume the $E_6$ singularity is
  at $[0:0:1]$ and the two flex points are at $[0:1:0]$ and
  $[1:0:0]$. The group $G$ fixes these three points, so $G$ is a
  subgroup of the torus
  \[\begin{pmatrix} a & & \\ & b & \\ & & c\end{pmatrix} \subset
    PGL(V).\]

  In addition, $G$ fixes the tangent line $L$ to the singularity. $L$
  meets the curve to order 4 at $[0:0:1]$, so it cannot intersect
  $[0:1:0]$ or $[1:0:0]$. Since $G$ must preserve $L$, it follows that
  $a=b$.

  Let $L_1$ be the flex line of $C$ at $[0:1:0]$ and $L_2$ be the
  analogue at $[1:0:0]$. By Bezout's theorem we know neither line
  $L_{i}$ contains $[0:0:1]$, and also $L_1\neq L_2$.  Therefore,
  either $L_1$ does not pass through $[0:1:0]$ or $L_2$ does not pass
  through $[1:0:0]$. In the first case, we find $a=c$ and in the
  second case we find $b=c$. In any case, $G$ is a trivial group and
  we are done.
\end{proof}

\begin{rmk}
  Note that in \Cref{E6} we don't assume the two flexes are ordinary
  flexes, i.e. not hyperflexes.  The conclusion ends up forcing this
  to be the case, however.
\end{rmk}

Now, let us assume we have a family
\[\alpha_{\on{h.flex}} : U \to \P \on{Sym}^{4}V^{\vee}\] which obeys
the following properties:
\begin{enumerate}
\item The family of curves $\pi: \mathcal{C} \to U$ is a smooth
  morphism,
\item the curve $C_{u}$, $u \in U$ general, does not have any
  hyperflex,
\item the curve $C_{0}$ has a hyperflex at $p$,
\item the family $\pi$ has two sections
  $\sigma_{1,2}: U \to \mathcal{C}$ meeting transversely at
  $p = \sigma_{1}(0) = \sigma_{2}(0)$ and such that $\sigma_{i}(u)$ is
  a flex point of $C_{u}$ for $u \in U$ general.
\end{enumerate}

In order to extract the suitable $1$-parameter family $\gamma$, we
will essentially perform an ``elementary transformation,''
$\P V \dashrightarrow \P V$, though the reader need not be familiar
with this concept.  

Let
\[\varphi: \mathcal{C} \to \P V\] denote the natural map, and let
$\mathcal{L}$ denote the line bundle $\varphi^{*}\mathcal{O}(1)$.
Suppose we have selected homogeneous variables so that the homogeneous
coordinate $X$ vanishes with order $4$ at $p$ along $C_{0}$, the
homogeneous coordinate $Y$ vanishes with multiplicity $1$ at $p$
(i.e. $Y=0$ is transverse to $C_{0}$ at $p$) and the homogeneous
coordinate $Z$ does not vanish at $p$. (Thus, $p = [0:0:1]$, and $X=0$
is the hyperflex line of $C_0$.) Abusing notation, we let $X, Y, Z$
denote the corresponding pulled back sections of $\mathcal{L}$ as
well.

Now, let $\mathcal{C}'$ denote the blow up of $\mathcal{C}$ at the
point $p$, and let \[\pi_{1}:\mathcal{C}' \to \mathcal{C}\] be the
blow-down map with exceptional curve $E$.  $E$ is a rational curve,
and the family of curves $\mathcal{C}' \to U$ has a two-component
fiber over $0$, namely $C_{0} \cup E$, with the two curves
intersecting transversely the point $p = C_{0} \cap E$.

Consider the line bundle
\[\mathcal{L}' := \pi_{1}^{*}\mathcal{L}(-4E).\]
Then $\pi_{1}^{*}X, \pi_{1}^{*}Y,$ and $\pi_{1}^{*}Z$ are rational
sections of $\mathcal{L}'$ with pole orders $0$, $3$, and $4$,
respectively along $E$.  Therefore, the sections
\[\pi_{1}^{*}X, t^{3}\pi_{1}^{*}Y. t^{4}\pi_{1}^{*}Z\] are regular
sections of $\mathcal{L}'$ on $\widetilde{\mathcal{C}}$, and they
vanish with multiplicity $0, 3, 4$ respectively along $C_{0}$. None of
them vanish entirely on $E$, and therefore, on $E$ they vanish with
orders $0$, $3$, and $4$ at the point $p \in E$.

Now consider the (regular) map
\begin{align}
  \label{newmap}
  \varphi': \widetilde{\mathcal{C}} \to \P^{2}
  \end{align}
  given by $[\pi_{1}^{*}X: t^{3}\pi_{1}^{*}Y: t^{4}\pi_{1}^{*}Z]$.  By
  what we have just said,
\begin{enumerate}
\item $\varphi'(C_0) = [1:0:0]$
\item the degree $4$ map $\varphi'|_{E}: E \to \P^{2}$ is birational
  onto its image $\overline{E}$: this is because the section with
  vanishing order $3$ at $p$ could not exist otherwise, as $3$ is
  relatively prime to $4$.
\item $\overline{E}$ has a uni-branched triple point at the point
  $ \varphi'|_{E}(p)$. (No other point $q \neq p$ can map to this
  point because the section $t^{4}\pi_{1}^{*}Z$ would vanish at a
  scheme of length $\geq 5$ on $E$, meaning it would be identically
  zero, which it isn't.)
\end{enumerate}

Therefore, the curve $\overline{E} \subset \P^{2}$ is a quartic with
an $E_{6}$ singularity at $\varphi'(p)$.  $\overline{E}$ can have no
other singularity, otherwise Bezout's theorem is violated.

If $\sigma_{i}': U \to \widetilde{\mathcal{C}}$, $i=1,2$ denote the
proper transforms of the sections $\sigma_{i}$, we note that by our
assumptions on the family $\alpha_{\on{h.flex}}$,
$p_{1} := \sigma_{1}'(0)$ and $p_{2} := \sigma_{2}'(0)$ are distinct
points of $E$, both distinct from $p$, and furthermore their
corresponding points $\overline{p}_{1}, \overline{p_{2}} \in E$ are
flex points. Therefore, by \Cref{E6}, we conclude that $\overline{E}$
is projectively equivalent to the curve $C_{E_{6}}$.

\begin{lem}
  \label{lem:E6}
  Maintain the setting in the discussion immediately prior.  If
  \[\gamma: U^{\times} \to PGL(V)\] is the $1$-parameter family
  defined by
  \[\begin{pmatrix} 1 & 0 & 0 \\
      0 & t^3 & 0\\
      0 & 0 & t^4\\
    \end{pmatrix}\] then $\gamma$ is full for $\alpha_{\on{h.flex}}$
  and $\alpha_{\on{h.flex}}^{\gamma}(0)$ is a curve projectively
  equivalent to $C_{E_{6}}$.
\end{lem}

\begin{proof}
  The map $\varphi': \mathcal{C}' \to \P^{2}$ from \eqref{newmap}
  resolves the composite
  \[\mathcal{C} \to \P V \dashrightarrow \P V\] where the first map is
  $\varphi$ and the second rational map sends a point $[x:y:z]$ to
  $[x:t^3y:t^4z]$.  Therefore, with $\gamma$ as stated in the lemma,
  the fact that $\overline{E} = C_{E_{6}}$ completes the proof.
\end{proof}

\begin{rmk}
  \label{rmk:E6} Not that the limiting endomorphism $\gamma(0)$ has
  kernel space equal to the flex line $X=0$ of $C_0$.
\end{rmk}

\begin{rmk}
  \label{rmk:generalE6}
  Condition (4) imposed on our family $\alpha_{\on{h.flex}}$ is not
  strictly necessary.  Possibly after performing a base change, and
  then blowing up $\mathcal{C}$ iteratively until the sections
  $\sigma_{i}$ become separated, we arrive at a specialization to the
  curve $C_{E_{6}}$.  We omitted this for brevity.
\end{rmk}

\begin{thm}
\label{allsmooth}
Let $C$ be a smooth quartic plane curve with $n$ hyperflexes. Then,
\begin{align*}
 \p_C = 8\p_{C_{A_6}}-n\p_{C_{E_6}}.
\end{align*}
\end{thm}

\begin{proof}
  We consider a family of smooth quartic curves, where the general
  member $C'$ has no hyperflexes, and with $C$ as the special fiber.

  Applying \Cref{lem:E6}, once per hyperflex, (also see
  \Cref{rmk:generalE6}) and using \Cref{principle} (\Cref{rmk:E6}
  shows that the hypotheses are met), we see that it suffices to check
  that the predegree of $C'$ is the predegree of $C$ plus $n$ times
  the predegree of $C_{E_6}$. The predegree of $C$ is $294n$ less than
  the predegree of $C'$ \cite[Section 3.6]{AF93}. Also, the predegree
  of $C_{E_6}$ is 294 from \cite[bottom of page 36]{AF00} or
  \Cref{A6E6orbit} (noting $\#\on{Aut}(C_{E_6})=2$).

  Finally, we use $\p_{C'}=8\p_{C_{A_6}}$ from \Cref{smooth}.
\end{proof}


\appendix
\section{Points on $\mb{P}^{1}$}
\label{sec:pointsonline}
This appendix will showcase two computations of $\p_X$ in the case $X$
is a hypersurface in $\mathbb{P}^1$, i.e. a set of points with
multiplicities.  In the case $X\subset \mathbb{P}^1$ is supported on
at most three points, these are strata of coincident root loci, which
were first computed in \cite{FNR06} and generalized to
$PGL_2$-equivariant cohomology in \cite{PGL2}. Therefore, we only have
to deal with the case where $X$ is supported on at least four points,
and we give two separate proofs. For formulas with fewer signs, and
also because our conventions are opposite to the conventions in
\cite{FML}, we will actually compute the class $\p_X(-u,-v)$ instead
of $\p_X(u,v)$.  Here $u,v$ are the chern roots of the universal rank
$2$ vector bundle $V$ over $\mathbb{B} GL(V)$.
\begin{thm}
\label{points}
Let $X\subset\mathbb{P}^1$ be a subscheme of length
$d = m_1 + \ldots + m_n$ supported on points $p_1,\ldots,p_n$ with
multiplicities $m_1,\ldots,m_n$ and with $n\geq 3$. Then,
\begin{align*}
  \p_X(-u,-v) = \frac{\prod_{i=0}^{d}{(iu+(d-i)v)}}{(u-v)^2}\left(\frac{n-2}{duv}+\sum_{i=1}^{n}{\frac{2m_i-d}{(m_iv+(d-m_i)u)(m_iu+(d-m_i)v)}}\right)
\end{align*}
\end{thm}

We give a proof of \Cref{points} using the resolution given by Aluffi
and Faber \cite{AFpoint} together with the Atiyah-Bott formula
\cite{EG98b} in \Cref{AB}. This proof is self-contained and
direct. The second proof we give is from the machinery developed in
\cite{FML} that apply to arbitrary hyperplane arrangements. A
computation is required to specialize the results from the case of
ordered points on $\mb{P}^1$ to unordered points on $\mb{P}^1$, we do
this now.
\begin{proof}[Proof using \cite{FML}]
  We use the same argument in \cite[Theorem 12.5]{FML}, so we only describe the computation, and refer the motivation and proof of correctness to \cite{FML}. Because our sign convention is opposite that of \cite{FML}, we will actually compute $\p_X(-u,-v)$. Let $d=\sum_{i=1}^{n}{m_i}$ and $G(z)=\prod_{i=0}^{d}{(H+iu+(d-i)v)}\in \mathbb{Z}[u,v][z]$. Let $L(z)=\frac{G(z)-G(0)}{z}$. Let $\overline{L}(H_1,\ldots,H_n)$ be the result of reducing $L(m_1H_1+\cdots+m_nH_n)$ modulo $(H_i+u)(H_i+v)$ for each $i$. Now, we carry out the three steps in the proof of \cite[Theorem 12.5]{FML}.\\
  \textbf{Step 1} By Lagrange interpolation,
\begin{align*}
    L(m_1H_1+\cdots+m_nH_n)&=\frac{G(m_1H_1+\cdots+m_nH_n)-L(0)}{m_1H_1+\cdots+m_nH_n}\\
    \overline{L}(H_1,\ldots,H_n)&=\sum_{T\subset \{1,\ldots,n\}}\frac{-G(0)}{-\sum_{i\in T}{m_iv}-\sum_{i\notin T}{m_i u}}\left(\prod_{i\in T}\frac{H_i+u}{-v+u}\right)\left(\prod_{i\notin T}\frac{H_i+v}{-u+v}\right).
\end{align*}\\
\textbf{Step 2} Substituting $z$ for each $H_i$ yields
\begin{align}
    \overline{L}(z,\ldots,z) &= G(0)\sum_{T\subset \{1,\ldots,n\}}\frac{1}{\sum_{i\in T}{m_iv}+\sum_{i\notin T}{m_i u}}\frac{(z+u)^{\# T}(z+v)^{d-\#T}}{\prod_{i\in T}{(-v+u)}\prod_{i\notin T}{(-u+v)}}.\label{z}
\end{align}
\textbf{Step 3} Let $F(z)=(z+u)(z+v)$. All terms of \eqref{z} are
divisible by $F(z)^2$ unless $\# T\in \{0,1,n-1,n\}$. Thus,
$[z^1][F(z)^1]\overline{L}(z,\ldots,z)$ is
\begin{align*} 
\frac{G(0)}{(u-v)^n}[z^1][F(z)^1]\left(\frac{(-1)^n(z+v)^n}{du}+\frac{(z+u)^n}{dv}+\sum_{i=1}^{n}{\frac{(-1)^{n-1}F(z)(z+v)^{n-2}}{m_iv+(d-m_i)u}}+\sum_{i=1}^{n}{\frac{(-1)F(z)(z+u)^{n-2}}{m_iu+(d-m_i)v}}\right). 
\end{align*}
As in the proof of \cite[Theorem 12.5]{FML},
\begin{align*}
  [z^1][F(z)^1]F(z)(z+u)^k &= (u-v)^{k-1} \qquad [z^1][F(z)^1](z+u)^k=(k-2)(u-v)^{k-3}\\
  [z^1][F(z)^1]F(z)(z+v)^k &= (v-u)^{k-1} \qquad [z^1][F(z)^1](z+v)^k=(k-2)(v-u)^{k-3},    
\end{align*}
so $[z^1][F(z)^1]\overline{L}(z,\ldots,z)$ simplifies to
\begin{align*}
  &\frac{G(0)}{(u-v)^n}\left(\frac{(-1)^n(n-2)(v-u)^{n-3}}{du}+\frac{(n-2)(u-v)^{n-3}}{dv}+\sum_{i=1}^{n}\frac{(-1)^{n-1}(v-u)^{n-3}}{m_iv+(d-m_i)u}+\sum_{i=1}^{n}\frac{(-1)(u-v)^{n-3}}{m_iu+(d-m_i)v}\right)\\
  &\frac{G(0)}{(u-v)^n}\left(\frac{(-1)(n-2)(u-v)^{n-3}}{du}+\frac{(n-2)(u-v)^{n-3}}{dv}+\sum_{i=1}^{n}{\frac{(u-v)^{n-3}}{m_iv+(d-m_i)u}+\frac{(-1)(u-v)^{n-3}}{m_iu+(d-m_i)v}}\right)\\
  &\frac{G(0)}{(u-v)^n}\left(\frac{(n-2)(u-v)^{n-2}}{duv}+\sum_{i=1}^{n}{\frac{(2m_i-d)(u-v)^{n-2}}{(m_iv+(d-m_i)u)(m_iu+(d-m_i)v)}}\right)\\
  &\frac{G(0)}{(u-v)^2}\left(\frac{n-2}{duv}+\sum_{i=1}^{n}{\frac{2m_i-d}{(m_iv+(d-m_i)u)(m_iu+(d-m_i)v)}}\right)
\end{align*}
\end{proof}

In the case all the multiplicities are all one, the formula in
\Cref{points} simplifies. We will also give a direct proof by slow
projection immediately after this corollary.
\begin{cor}
  \label{multiplicityone}
In the setting of \Cref{points} if each $m_i=1$, then $n=d$, and
\begin{align*}
  \mathsf{P}_X(-u,-v) = n(n-1)(n-2)\prod_{j=2}^{n-2}{(H+(ju+(n-j)v))}.
\end{align*}
\end{cor}

\begin{proof}[Proof using \Cref{points}]
  Applying \Cref{points}, we find $\p_X(-u,-v)$ is
\begin{align*}
\frac{1}{(u-v)^2}\prod_{i=0}^{n}{(iu+(n-i)v)}\left(\frac{n-2}{nuv}-\frac{(-2+n)n}{((n-1)u+v)((n-1)v+u)}\right)&=\\
\frac{n(n-2)}{(u-v)^2}\prod_{i=0}^{n}{(iu+(n-i)v)}\left(\frac{1}{(nu)(nv)}-\frac{1}{((n-1)u+v)((n-1)v+u)}\right)&=\\
\frac{n(n-2)}{(u-v)^2}\prod_{i=0}^{n}{(iu+(n-i)v)}\left(\frac{n-1}{(nu)((n-1)u+v)((n-1)v+u)(nv)}\right)&.
\end{align*}
Applying \cite[Theorem 6.1]{FNR05} yields the result.
\end{proof}

\begin{proof}[Proof by ``slow projection'']
  Let $V$ be a $2$-dimensional vector space, $v_1,\ldots,v_n$ pairwise
  linearly independent vectors of $V$, $X\subset\mathbb{P}^1$ the
  corresponding point configuration supported on $p_1,\ldots,p_n$, and
  $Z\subset \mathbb{P}(\on{Sym}^nV)$ the orbit closure. The key fact
  we will use is
\begin{clm}
\label{22}
Every point in the boundary of $Z$ corresponds to a point
configuration in $\mathbb{P}^1$ supported on two points with
multiplicities $n-1$ and 1 or one point with multiplicity $n$.
\end{clm}

\begin{proof}[Proof of \Cref{22}]
  Let $A(t)$ be a 1-parameter family of matrices, or more precisely a
  map from the spectrum of a discrete valuation ring to $\on{End}(V)$
  where the generic point maps to an element of $GL(V)$. We want to
  show that the multiset
  $S=\{\lim_{t\to 0}A(t)p_i\mid 1\leq i\leq n\}$ does not have two
  copies each of two distinct points. First, we can assume the rank of
  $A(0)=1$. If the rank of $A(0)$ is 2, then $S$ consists of distinct
  points. If $A(0)=0$, we can divide out by a power of the
  uniformizing parameter so that $A(0)\neq 0$.  Then,
  $\{\lim_{t\to 0}A(t)p_i\mid 1\leq i\leq n\}$ is the point in
  $\mb{P}(V)$ corresponding to the $1$-dimensional image of $A(0)$ if
  $v_i$ is not in the kernel of $A(0)$. Otherwise, there is at most
  one $v_i$ in the kernel of $A(0)$ and
  $\{\lim_{t\to 0}A(t)v_i\mid 1\leq i\leq n\}$ is otherwise
  unrestricted.
\end{proof}

Let $x,y$ be a basis for $V$. Then, a basis of $\on{Sym}^nV$ is
$x^n,x^{n-1}y,\ldots,y^n$. Let $T\subset GL(V)$ be the maximal torus
corresponding to the basis $x,y$. Since
$A^{\bullet}_{GL(V)}(\mathbb{P}(\on{Sym}^nV))\to
A^{\bullet}_T(\mathbb{P}(\on{Sym}^nV))$ is injective, we can use a
$T$-equivariant degeneration and compute the $T$-equivariant
class. Our $T$-equivariant degeneration will be to scale the
coordinates corresponding to $x^{n-2}y^2,\ldots,x^2y^{n-2}$ to zero.

By \Cref{22}, $Z$ is disjoint from the source of this ``slow
projection,'' so the $T$-equivariant class of $Z$ is a multiple of the
class of the $3$-plane in $\mathbb{P}(\on{Sym}^nV)$ given by the
vanishing of the coordinates corresponding to
$x^{n-2}y^2,\ldots,x^2y^{n-2}$. The class of that $3$-plane is
$\prod_{i=2}^{n-2}{(iu+(n-i)v)}$. The multiple we need is the degree
of $Z$ as a projective variety, which is $n(n-1)(n-2)$ by the
combinatorial argument given in \cite[Introduction]{AFpoint}.
\end{proof}

\begin{rmk}
  \Cref{multiplicityone} can be generalized in a different
  direction. Suppose we fix $d$ general points
  $p_1,\ldots,p_d\in \mathbb{P}^r$ and consider all configurations of
  $d$ points given by mapping $p_1,\ldots,p_d$ via a linear rational
  map $\mathbb{P}^r\to \mathbb{P}^1$. Let the closure of these
  configurations in $\on{Sym}^d\mathbb{P}^1$ be $Z_{r,d}$. The same
  proof of \Cref{22} using slow projection shows the equivariant class
  of $Z_{r,d}$ in
  $A^{\bullet}_{GL(V)}(\on{Sym}^d\mathbb{P}^1)=\mathbb{Z}[u,v][H]/(\prod_{i=0}^{n}{H+iu+(n-i)v})$
  has constant term
\begin{align*}
    \prod_{i=0}^{2r+1}{(d-i)}\prod_{i=r+1}^{d-r-1}{(iu+(d-i)r)},
\end{align*}
and the full class is given by substituting
$u\to u+\frac{H}{d},v\to v+\frac{H}{d}$ into the constant term
\cite[Theorem 6.1]{FNR05}. These are examples of \emph{generalized
  matrix orbits} defined in \cite{FML}. Also see \cite[Example
1.3]{T19}.
\end{rmk}
\section{Points on $\mathbb{P}^1$ via Atiyah-Bott}
\label{AB}
The method in \Cref{sec:pointsonline} was closer to the theme of
equivariant degeneration explored in this paper. We note that there is
self-contained proof given by the Atiyah-Bott formula, or equivalently
resolution and integral via localization \cite[Section 4]{FR06}. The
authors attempted to perform the same method for smooth plane curves
using the resolution given by \cite{AF93}, but the computation of the
normal bundles quickly became intractable.

\subsection{General setup}
Let $V$ be a $2$-dimensional vector space with $T=(\mb{C}^{\times})^2$
acting by scaling. Then, we have $T$-action on
$\mb{P}^3\cong \mb{P}{\rm Hom}(V,\mb{C}^2)$. Given a point
configuration of $d$-points in $\mb{P}^1$ (a central hyperplane
configuration in $\mb{C}^2)$, we have a rational map
\begin{align*}
\mb{P}{\rm Hom}(V,\mb{C}^2)\dashrightarrow \mathbb{P}(\on{Sym}^dV)
\end{align*}
Suppose our point configuration consists of $n$ distinct points $p_1,\ldots,p_n$ with multiplicities $m_1,\ldots,m_n$ summing to $d$. Then, the base locus is $n$ disjoint lines, given by matrices with image contained in each $p_1,\ldots,p_n$. Picking a basis, we find
$\mb{P}{\rm Hom}(V,\mb{C}^2)$ is given by 2 by 2 matrices
$\begin{pmatrix} \ast & \ast \\ \ast & \ast \end{pmatrix}$ up to
scaling. The $T$ action is by scaling the columns, and each of the
base loci (after base change via a left $GL_2$-action) looks like
$\begin{pmatrix} \ast & \ast \\ 0 & 0\end{pmatrix}$.

Let $X$ be the blow up of $\mb{P}^3$ along these base loci
$R_1,\ldots,R_n$. This resolves the rational map above
\cite[Proposition 1.2]{AF93}.

\subsection{Normal bundle to a proper transform}
\begin{lem}
\label{SES}
Let $Z\subset Y$ be an inclusion of smooth varieties. Let $W\subset Y$
be a smooth subvariety and $\wt{W}\subset {\rm Bl}_Z Y$ be the proper
transform of $W$. If $\pi: {\rm Bl}_Z Y\to Y$ is the blowup map, then
we have the short exact sequence
\begin{align*}
  0\to \on{coker}(\pi^{*}N_{W/Y}^\vee \to N_{\wt{W}/{\rm Bl}_Z Y}^\vee)\to \Omega_{{\rm Bl}_Z Y/Y}|_{\wt{W}}\to  \Omega_{\wt{W}/W}\to 0. 
\end{align*}
\end{lem}

\begin{proof}
Consider the following diagram
\begin{center}
\begin{tikzcd}
  & 0 \arrow[d] & 0 \arrow[d] & &\\
  & \pi^{*}N_{W/Y}^\vee \arrow[r] \arrow[d] & N_{\wt{W}/{\rm Bl}_Z Y}^\vee  \ar[d] &  &\\
  0 \arrow[r] &  \pi^{*}\Omega_{Y}|_{\wt{W}} \arrow[r] \ar[d] &  \Omega_{{\rm Bl}_Z Y}|_{\wt{W} } \arrow[r] \ar[d] &  \Omega_{{\rm Bl}_Z Y/Y}|_{\wt{W}} \arrow[r] \ar[d]& 0\\
  0 \arrow[r] & \pi^{*}\Omega_W \arrow[r]  \ar[d] & \Omega_{\wt{W}/W}  \arrow[r] \ar[d] & \Omega_{\wt{W}/W} \arrow[r] \ar[d] & 0\\
  & 0 & 0 & 0
\end{tikzcd}
\end{center}
The bottom two rows are exact by the relative cotangent sequence for a
generically separable morphism of integral smooth varieties
\cite[Remark 4.17]{Liu}. The lemma follows from the nine lemma.
\end{proof}

\subsection{Setup for Atiyah-Bott integration}
In order to apply Atiyah-Bott integration to $X$, we need to identify
the fixed loci, their normal bundles, and how classes restrict from
$X$ to the fixed point loci.

\subsection{Fixed point loci}
First, we note that the fixed-point loci of $\mb{P}^3$ under the
action of $T$ consists of two disjoint $\mb{P}^1$'s which we will call
$C_1,C_2$.
\begin{align*}
\begin{pmatrix}
\ast & 0\\
\ast & 0
\end{pmatrix}
\qquad
\begin{pmatrix}
0 & \ast \\
0 & \ast
\end{pmatrix}
\end{align*}
The fixed-point loci of $X$ under the action of $T$ must lie over the
fixed-point loci of $\mb{P}^3$ under $T$. Therefore, we conclude that the fixed loci consist of the following 2n+2 components:
\begin{enumerate}
\item 2 fixed point loci corresponding to $\mb{P}^1$'s that are the
  proper transforms $\wt{C_1}$ and $\wt{C_2}$ of $C_1$ and $C_2$. If we suppose $R_1$ is the $\mathbb{P}^1$ consisting of the matrices
  $\begin{pmatrix} \ast & \ast \\ 0 & 0\end{pmatrix}$, then the point
  of the proper transform lying above $C_1\cap R_1$ is given by the
  limiting point of
\begin{align*}
\begin{pmatrix} 
1 & 0\\
0 & 0
\end{pmatrix} 
+ t 
\begin{pmatrix} 
0 & 0\\
1 & 0
\end{pmatrix}
\end{align*}
as $t\to 0$. 

\item $2n$ isolated points that lie over the $2n$ pairwise
  intersections of $C_1,C_2$ with $R_1,\ldots,R_n$. If we suppose
  $R_1$ is the $\mathbb{P}^1$ consisting of the matrices
  $\begin{pmatrix} \ast & \ast \\ 0 & 0\end{pmatrix}$, then the
  isolated torus-fixed point lying above $C_1\cap R_1$ is given by the
  limiting point of
\begin{align*}
\begin{pmatrix} 
1 & 0\\
0 & 0
\end{pmatrix} 
+ t 
\begin{pmatrix} 
0 & 0\\
0 & 1
\end{pmatrix}
\end{align*}
as $t\to 0$. 
\end{enumerate}

\subsubsection{Normal bundles and restriction of proper transforms}
Let $H$ be the $c_1(\ms{O}_{\mb{P}^3}(1))$ on $\mb{P}^3$ pulled back
to $X$ and $E$ an exceptional divisor of $X\to \mb{P}^3$. We have
$\wt{C}_1$ is $\mb{P}^1$ with a trivial T-action. Therefore,
$A^{\bullet}(\wt{C}_1)\cong \mb{Z}[z][u,v]/(z^2)$. We have the
following restrictions:
\begin{align*}
H \mapsto H=z-u\\
E\mapsto z = H+u. 
\end{align*}
Here, $E$ is any of the $n$ exceptional divisors.  (We are thinking of
$\wt{C}_1$ as the $\mb{P}^1$ embedded as the first column of 2 by 2
matrices $\mb{P}^3$ up scaling. Therefore, it's actually natural to
think of it as the projectivization of a vector bundle with a
nontrivial $T$-action, so it is a projective bundle over a point that
is trivial, but $\ms{O}(1)=H$ is twisted. The Leray relation in this
case is $(H+u)^2=z^2$.)

We need to compute the normal bundle to the proper transform of
$C_1$. The normal bundle of $C_1$ in $\mb{P}^3$ is
\begin{align*}
\frac{c(\mb{P}^3)}{c(C_1)}=\frac{(1+u+H)^2(1+v+H)^2}{(1+u+H)^2}=(1+v+H)^2.
\end{align*}
Note that this also makes sense as $C_1$ is a complete intersection
cut out by $(v+H)^2$.
Applying Lemma \ref{SES} yields 
\begin{align*}
0\to \pi^{*}N_{C_1/\mb{P}^3}^\vee\to N_{\wt{C}_1/X}^\vee\to \Omega_{X/\mb{P}^3}|_{\wt{C}_1}\to 0.
\end{align*}
The term on the right is a skyscraper sheaf supported on the
intersection of $\wt{C}_1\cong \mb{P}^1$ with the exceptional
locus. We need to find the torus action on the bundle
$T_{X/\mb{P}^3}|_{\wt{C}_1}$ supported on $E$ at the intersection
$E\cap \wt{C}_1$. There is an affine neighborhood of $E\cap \wt{C}_1$
in $X$ of the form
\begin{align*}
\begin{pmatrix}
1 & \frac{a_{01}}{a_{00}}\\
\frac{a_{10}}{a_{00}} & \frac{a_{11}}{a_{00}}
\end{pmatrix}
+ t
\begin{pmatrix}
0 & 0 \\
1 & \frac{A_{11}}{A_{10}}
\end{pmatrix}
\end{align*}
with coordinates given by $\frac{a_{01}}{a_{00}}$,
$\frac{a_{10}}{a_{00}} $, $\frac{A_{11}}{A_{10}}$ as
$\frac{A_{11}}{A_{10}}a_{10}=a_{11}$. We have the short exact sequence
\begin{align*}
0\to \wt{C}_1(-z)\otimes \mb{C}_{v-u}\to \wt{C}_1\otimes \mb{C}_{v-u}\to \mb{C}_{v-u}|_{\wt{C}}\cap \pi^{-1}(R_i)\to 0,
\end{align*}
where $\mb{C}_{v-u}$ is the nonequivariantly trivial line bundle with
an action of $T$ by the character $v-u$. The torus action has
character $v-u$ on the coordinate $\frac{A_{11}}{A_{10}}$, so the term
on the right has chern class $\frac{1+v-u}{1-z+v-u}$. We apply this
for each $i$ to find
\begin{align*}
c(N_{\wt{C}_1/X}) &= (1+H+v)^2\frac{(1-z+v-u)^n}{(1+v-u)^n}\\
&=(1+v-u)^2(1+\frac{z}{1+v-u})^2(1-\frac{z}{1+v-u})^n\\
&=(1+v-u)^2(1+\frac{(2-n)z}{1+v-u})\\
&=(1+z+v-u)(1+(1-n)z+v-u).
\end{align*}

\subsubsection{Restriction to isolated points}
Suppose one of our exceptional divisors is the blow up of the locus
$R$ consisting of matrices
$\begin{pmatrix} \ast & \ast \\ 0 & 0 \end{pmatrix}$ with image
contained in $[1:0]\in \mathbb{P}^1$. Then, as described in Section
B.4, there is an isolated torus-fixed point in the exceptional $E$
lying over the locus $R$ given as the limit as $t\rightarrow 0$ of the
one-parameter family
\begin{align*}
\begin{pmatrix}
1 & 0\\
0 & 0
\end{pmatrix}
+ t
\begin{pmatrix}
0 & 0\\
0 & 1 
\end{pmatrix}.
\end{align*}
Then, we have the restrictions
\begin{align*}
H&\mapsto -u\\
E&\mapsto v-u.
\end{align*}
Here, $E$ is the exceptional divisor containing $p$. The first one is
by restricting the tautological line bundle and considering the torus
action. To see the restriction of $E$, we note that the restriction of
$E$ to itself is $\ms{O}_{\mb{P}(N_{R_1/\mb{P}^3})}(-1)$. Then, we
take the local chart around $\pi(p)$ consisting of
\begin{align*}
\begin{pmatrix}
1 & \frac{a_{01}}{a_{00}}\\
\frac{a_{10}}{a_{00}} & \frac{a_{11}}{a_{00}}
\end{pmatrix}
\end{align*}
and find the action on the coordinate $\frac{a_{11}}{a_{00}}$ is
$v-u$. Also the normal bundle to $p$ in $X$ has chern class
\begin{align*}
(1+v-u)^2(1+u-v). 
\end{align*}
To see this, consider the local chart around $p$
\begin{align*}
\begin{pmatrix}
1 & \frac{a_{01}}{a_{00}}\\
0 & \frac{a_{11}}{a_{00}}
\end{pmatrix}
+ t
\begin{pmatrix}
0 & 0 \\
\frac{A_{10}}{A_{11}} & 1
\end{pmatrix}
\end{align*}
which has local coordinates $ \frac{a_{01}}{a_{00}}$,
$\frac{a_{11}}{a_{00}}$, $\frac{A_{10}}{A_{11}}$ on which $T$ acts by
characters $v-u$, $v-u$ and $u-v$ respectively.

\subsection{Application of Atiyah-Bott}

\begin{proof}[Proof of \Cref{points}]
  As before, we compute $\p_X(-u,-v)$ due to our sign conventions. Let
\begin{align*}
\phi(H) = \frac{(H+du)(H+(d-1)u+v)\cdots (H+dv)-(du)\cdots (dv)}{H}.
\end{align*}

We want to pull $\phi(H)$ back to $X$ and integrate using
Atiyah-Bott. We first integrate over $\wt{C_1}$. Since $H$ pulls back
to
$$dH-\sum_{i=1}^{n}{m_iE_i}=d(z-u)-dz=-du,$$
 this is
\begin{align*}
[z]\frac{1}{(z+v-u)((1-n)z+v-u)}\phi(-du)&=\\
[z]\frac{\frac{1}{(v-u)^2}}{(1+\frac{z}{v-u})(1+\frac{(1-n)z}{v-u})}\frac{-\prod_{i=0}^{d}{(iu+(d-i)v)}}{-du}&=\\
\frac{(n-2)\prod_{i=1}^{d}{(iu+(d-i)v)}}{(v-u)^3}&. 
\end{align*}
Adding this to the contribution of $\wt{C_2}$ yields
\begin{align*}
(n-2)\prod_{i=0}^{d}{(iu+(d-i)v)}\frac{1}{(v-u)^3}\left(\frac{1}{du}-\frac{1}{dv}\right)&=\\
(n-2)\prod_{i=0}^{d}{(iu+(d-i)v)}\frac{1}{(v-u)^2}\frac{1}{duv}
\end{align*}
For each $1\leq i\leq n$, we have a point in the configuration of
multiplicity $m_i$. We have two isolated fixed points corresponding to
$i$ lying above $R_i\cap C_1$ and $R_i\cap C_2$. For the point lying
above $R_i\cap C_1$, $H$ pulls back to $dH-m_iE$, where $E$ is the
exceptional divisor lying above $R_i$. This restricts to
\begin{align*}
-du-n(v-u)=(-d+m_i)u-m_iv
\end{align*}
at the fixed point. The contribution to Atiyah Bott is
\begin{align*}
\frac{1}{(u-v)^3}\phi((-d+m_i)u-m_iv)&=\\
\frac{1}{(u-v)^3}\frac{-\prod_{j=0}^{d}{(ju+(d-j)v)}}{(-d+m_i)u-m_iv}&.
\end{align*}
Adding this to the contribution of the fixed point lying above
$R_i\cap C_2$, we get
\begin{align*}
  \frac{1}{(u-v)^3}\prod_{j=0}^{d}{(ju+(d-j)v)}\left(\frac{1}{(d-m_i)u+m_iv}-\frac{1}{(d-m_i)v+m_iu}\right)&=\\
  -\frac{1}{(u-v)^2}\prod_{j=0}^{d}{(ju+(d-j)v)}\frac{d-2m_i}{((d-m_i)u+m_iv)((d-m_i)v+m_iu)}&.
\end{align*}
Adding the contributions up yields the desired result.
\end{proof}

\section{Cubic plane curves}
\label{sec:cubicplanecurves}

The computations of $\p_C$ for cubic plane curves $C$ are elementary,
but we provide them here for the sake of completeness.

The following table provides a complete list of polynomials classes
$[\Orb(C)]_{GL(V)}$ for all cubics.  When the automorphism group is
infinite, the entries are simply $[\Orb(C)]_{GL(V)}$.  When the
automorphism group is finite, the entries are $\p_{C}$.
\begin{center}
\begin{longtable}{l | p{8cm}|p{2cm}}
  Cubic Curve $C$ & $[\Orb(C)]_{GL(V)}$ or $\p_{C}$ &$\#\on{Aut}$\\\hline
  Triple Line & $-(72 c_1^3 c_2^2+36 c_1 c_2^3+36 c_1^4 c_3-162 c_1^2 c_2 c_3+243 c_1 c_3^2)$ & $\infty$\\
  Double Line plus Line &
                          $-(72 c_1^3 c_2 + 36 c_1 c_2^2 - 108 c_1^2 c_3)$& $\infty$\\
  Three concurrent lines &
                           $12 c_1^4 + 6 c_1^2 c_2 + 27 c_1 c_3$& $\infty$\\
  Conic plus tangent line & 
                            $ -36 c_1^3 - 18 c_1 c_2$& $\infty$ \\
  Triangle & 
             $-(12 c_1^3 + 6 c_1 c_2 + 27 c_3)$& $\infty$ \\
  Conic plus line & 
                    $18 c_1^2 + 9 c_2$& $\infty$ \\
  Cuspidal cubic & 
                   $24 c_1^2$& $\infty$ \\
  Irreducible nodal cubic & 
                            $(-12c_1)6$& 6\\ 
  Smooth cubic ($j \neq 0, 1728$) &
                                    $(-12c_1)18$&18 \\
  Smooth cubic with $j = 1728$ &
                                 $(-6c_1)36$&36\\
  Smooth cubic with $j = 0$ &
                              $(-4c_1)54$&54
\end{longtable}
\end{center}

Let us only indicate the methods of calculation, leaving details to
the reader.  The formulas for a triple line, double line plus line,
conic plus line, and triangle can all be obtained via presentation and
integration along the lines of \cite[Theorem 3.1]{FNR06} and
\Cref{fourlines}. This is the method of resolution and integration
\cite[Section 3]{FR06}.

The formula for three concurrent lines, conic plus tangent line and
cuspidal cubic can be gotten by applying Kazarian's formula
\cite[Theorem 1]{Kazarian} for counting $D_4$, $A_3$, and $A_2$
singularities respectively. This was carried out for the case of
quartic plane curves for $A_6$, $D_6$, and $E_6$ in the proof of
\Cref{A6E6orbit}. The formula for smooth and nodal cubics can be
obtained by their predegree formulas \cite[Section 3.6]{AF93} and
\Cref{projectivethom}.

\bibliographystyle{alpha}
\bibliography{references.bib}

\end{document}